\documentclass[oneside,reqno,a4paper,11pt]{amsart}
\usepackage{amsfonts}
\usepackage{amsmath,amsfonts,amssymb}
\usepackage{CJK,bbm,color,soul,supertabular,longtable,verbatim}
\usepackage[warning, easyold,math]{easyeqn}
\usepackage{graphicx}
\usepackage{color}
\usepackage{latexsym,amsfonts,mathrsfs, amsmath, amssymb, amscd, epsfig}
\usepackage{mathrsfs}
\usepackage{amsthm}
\usepackage{esint}
\newtheorem {theorem}{Theorem}[section]
\newtheorem {lemma}[theorem]{{\bf Lemma}}

\newtheorem {prop}[theorem]{{\bf Proposition}}
\theoremstyle{remark}
\newtheorem {remark}{{\bf Remark}}[section]

\theoremstyle{problem}

\theoremstyle{definition}
\newtheorem {definition}{{\bf Definition}}[section]
\theoremstyle{plain} \numberwithin {equation}{section}

\begin{document}
%\today
\vspace{1cm}
\title[Impinging outgoing jets]{Steady collision of two jets issuing from two axially symmetric channels$^\dag$}
\author[Lili Du, \ \ Yongfu Wang]{Lili Du$^{\lowercase {a,1}}$, \ \ Yongfu Wang$^{\lowercase {b,2}}$}
\thanks{$^\dag$ Du is supported by NSFC grant 11971331 and Sichuan Youth Science and Technology Foundation (No. 21CXTD0076). Wang is supported by NSFC grant 11801460 and the Applied Fundamental Research Plan of Sichuan Province (No.~2021YJ0520).}
\thanks{$^1$ E-Mail: dulili@scu.edu.cn. $^2$ E-mail:
wyf1247@163.com. Corresponding author}

\maketitle
\begin{center}
$^a$ Department of Mathematics, Sichuan University,

          Chengdu 610064, P. R. China.

$^b$ School of Economic Mathematics,

Southwestern University of Finance and Economics,

          Chengdu 611130, P. R. China.

\end{center}

\begin{abstract}In the classical survey (Chapter 16.2, {\it Mathematics in industrial problem}, Vol. 24, Springer-Verlag, New York, 1989), A. Friedman proposed an open problem on the collision of two incompressible jets emerging from two axially symmetric nozzles. In this paper, we concerned with the mathematical theory on this collision problem, and establish the well-posedness theory on hydrodynamic impinging outgoing jets issuing from two coaxial axially symmetric nozzles. More precisely, we showed that for any given mass fluxes $M_1>0$ and $M_2<0$ in two nozzles respectively, that there exists an incompressible, inviscid impinging outgoing jet with contact discontinuity, which issues from two given semi-infinitely long axially symmetric nozzles and extends to infinity. Moreover, the constant pressure free stream surfaces of the impinging jet initiate smoothly from the mouths of the two nozzles and shrink to some asymptotic conical surface. There exists a smooth surface separating the two incompressible fluids and the contact discontinuity occurs on the surface. Furthermore, we showed that there is no stagnation point in the flow field and its closure, except one point on the symmetric axis. Some asymptotic behavior of the impinging jet in upstream and downstream, geometric properties of the free stream surfaces are also obtained. The main results in this paper solved the open problem on the collision of two incompressible axially symmetric jets in \cite{FA3}.\end{abstract}
\

\begin{center}
\begin{minipage}{5.5in}

\

\end{minipage}
\end{center}

\

\everymath{\displaystyle}
\newcommand {\eqdef }{\ensuremath {\stackrel {\mathrm {\Delta}}{=}}}

\def\Xint #1{\mathchoice
{\XXint \displaystyle \textstyle {#1}} %
{\XXint \textstyle \scriptstyle {#1}} %
{\XXint \scriptstyle \scriptscriptstyle {#1}} %
{\XXint \scriptscriptstyle \scriptscriptstyle {#1}} %
\!\int}
\def\XXint #1#2#3{{\setbox 0=\hbox {$#1{#2#3}{\int }$}
\vcenter {\hbox {$#2#3$}}\kern -.5\wd 0}}
\def\ddashint {\Xint =}
\def\dashint {\Xint -}
\def\clockint {\Xint \circlearrowright } % GOOD !
\def\counterint {\Xint \rotcirclearrowleft } % Good for Computer Modern !
\def\rotcirclearrowleft {\mathpalette {\RotLSymbol { -30}}\circlearrowleft }
\def\RotLSymbol #1#2#3{\rotatebox [ origin =c ]{#1}{$#2#3$}}

\def\aint{\dashint}

\def\arraystretch{2}
\def\eps{\varepsilon}

\def\s#1{\mathbb{#1}} % set
\def\t#1{\tilde{#1}} %new variables
\def\b#1{\overline{#1}}
\def\N{\mathcal{N}} %Nozzle
\def\M{\mathcal{M}} %Mach number
\def\R{{\mathbb{R}}}
\def\B{{\mathcal{B}}}
\def\BB{\mathfrak{B}}
\def\F{{\mathcal{F}}}
\def\G{{\mathcal{G}}}
\def\ba{\begin{array}}
\def\ea{\end{array}}
\def\be{\begin{equation}}
\def\ee{\end{equation}}

\def\bes{\begin{mysubequations}}
\def\ees{\end{mysubequations}}

\def\cz#1{\|#1\|_{C^{0,\alpha}}}
\def\ca#1{\|#1\|_{C^{1,\alpha}}}
\def\cb#1{\|#1\|_{C^{2,\alpha}}}
\def\psir{\left|\frac{\nabla\psi}{r}\right|^2}
\def\lb#1{\|#1\|_{L^2}}
\def\ha#1{\|#1\|_{H^1}}
\def\hb#1{\|#1\|_{H^2}}
\def\th{\theta}
\def\Th{\Theta}
\def\cin{\subset\subset}
\def\Ld{\Lambda}
\def\ld{\lambda}
\def\ol{{\Omega_L}}
\def\sla{{S_L^-}}
\def\slb{{S_L^+}}
\def\e{\varepsilon}
\def\C{\mathbf{C}} %%unit cylinder
\def\ra{\rightarrow}
\def\xra{\xrightarrow}
\def\g{\nabla}
\def\a{\alpha}
\def\b{\beta}
\def\d{\delta}
\def\th{\theta}
\def\fai{\varphi}
\def\O{\Omega}
\def\ol{{\Omega_L}}
\def\psirk{\left|\frac{\nabla\psi}{r+k}\right|^2}
\def\tO{\tilde{\Omega}}
\def\tu{\tilde{u}}
\def\tv{\tilde{v}}
\def\trho{\tilde{\rho}}
\def\W{\mathcal{W}}
\def\f{\frac}
\def\p{\partial}
\def\o{\omega}
\def\B{\mathcal{B}}
\def\H{\Theta}
\def\msS{\mathscr{S}}
\def\bq{\mathbf{q}}
\def\msE{\mathcal{E}}
\def\mfa{\mathfrak{a}}
\def\mfb{\mathfrak{b}}
\def\mfc{\mathfrak{c}}
\def\mfd{\mathfrak{d}}
\def\ff{\texthtbardotlessj}
\def\mcL{\mathcal{L}}
\def\mcR{\mathcal{R}}
\def\Div{\text{div}}
\section{Introduction}

The three-dimensional incompressible, stationary and inviscid flow is governed by the Euler equations
\be\label{a1} \left\{ \ba{l}
\sum_{i=1}^3\f{\p u_i}{\p x_i} =0,\\
\sum_{i=1}^3u_i\f{\p u_j}{\p x_i} + \f{1}{\rho}\f{\p P}{\p x_j}=0,\quad
\text{for}\quad j=1,2,3,
 \ea \right.\ee with the irrotational condition
$$\g \times (u_1,u_2,u_3)=0.$$
Here, $(u_1,u_2,u_3)$ is the velocity, $P$ denotes the pressure
of the flow, and positive constant $\rho$ denotes density.

In this paper, we shall be concerned with steady, irrotational
incompressible impinging jets issuing from two semi-infinitely long
axially symmetric nozzles with variable cross-section. We will investigate
the well-posedness theory of the collision problem of two jets
issuing from two general axially symmetric nozzles, and solve the
open problem (1) proposed by A. Friedman in 1989.

Consider the axially symmetric flows in this paper and
let $U(x,r)$, $V(x,r)$ and $W(x,r)$ be the axial velocity, the
radial velocity and the swirl velocity respectively, $x=x_1$ and
$r=\sqrt{x_2^2+x_3^2}$. Furthermore, we seek such an axially
symmetric flow without swirl, one has \be\label{a2} u_1=U(x,r),\quad
u_2=V(x,r)\f{x_2}{r},\quad u_3=V(x,r)\f{x_3}{r}.\ee Then, instead of
\eqref{a1}, we have
 \be\label{a3}\left\{\ba{l}
\left(r U\right)_{x}+\left(r V\right)_{r}=0,\\
 \left(r\rho U^2\right)_{x}+\left(r\rho UV\right)_{r}+rP_{x}=0,\\
 \left(r\rho UV\right)_{x}+\left(r\rho V^2\right)_{r}+rP_{r}=0.\\
\ea\right.
 \ee

Consider the flow issuing from the two semi-infinitely long nozzles
as (see Figure \ref{f01})
$$
\mathcal{N}_1=\left\{(x,r)\in \R_+^2\left|f_1(r) < x<-1,\
 r<R_1\right.\right\},$$ and
$$
\mathcal{N}_2=\left\{(x,r)\in \R_+^2\left|1<x<f_2(r),\
r<R_2\right.\right\},$$ where $\R_+^2=\R^1\times[0,+\infty)$, $R_1$, $R_2>0$,
$f_1(r)$ and $f_2(r)$ are smooth functions and satisfy that
 \be\label{a4} f_i(r)=(-1)^i\infty, \quad\quad r\leq r_i,\quad (r_i>0)\ee
and \be\label{a5} f_i(R_i)=(-1)^i, \ee for $R_i>r_i$ and $i=1,2$.
Without loss of generality, we assume that $R_1=R_2=R$.
%Additionally, we assume that
%\be\label{a03}\text{$f_1(r)\leq\f{2(r-R_1)}{R_2-R_1}-1$\quad\quad
%for\quad $R_2<r\leq R_1$}.\ee
\begin{figure}[!h]
\includegraphics[width=100mm,height=40mm]{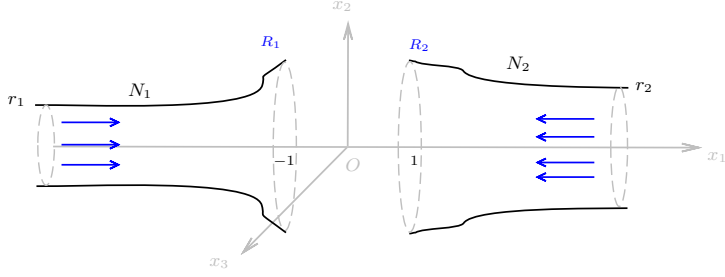}
\caption{Two axially symmetric semi-infinitely long nozzles}\label{f01}
\end{figure}

For convenience, we denote the symmetric axis
$$N_0=\{(x,r)|r=0,-\infty<x<+\infty\},$$
 the nozzle walls $$N_1=\{(x,r)|x=f_1(r),\ r_1<r<R\},\quad
N_2=\{(x,r)|x=f_2(r),\ r_2<r<R\},$$ and the edge points of the
nozzle walls $A_1=(-1,R)$ and $A_2=(1,R)$.

%We remark that the condition \eqref{a03} means that the nozzle wall $N_1$
%lies below the line segment $\overline{A_1A_2}$.

In this paper, we consider two ideal, nonmiscible, irrotational fluids $(U_1, V_1,
P_1, \rho_1)$ and $(U_2, V_2, P_2, \rho_2)$ issuing from two semi-infinite axisymmetric nozzles. We designate by $(U_1, V_1,
P_1, \rho_1)$ and $(U_2, V_2, P_2, \rho_2)$ be the axial velocity, the radial velocity and the pressure of the fluid I and the fluid II, respectively. Denote $\O_i$ as the fluid field of the $i$-th fluid for $i=1,2$, and  $$(U,V,P, \rho)=\left\{ \ba{l}
(U_1, V_1, P_1,\rho_1)\ \ \ \text{in}\ \ \O_1,\\
(U_2, V_2, P_2,\rho_2)\ \ \ \text{in}\ \ \O_2,
 \ea \right.\text{as the two-phase fluids}.$$

 In this paper, we seek a contact discontinuity $(U, V, P, \rho)$ with a smooth interface $\Gamma$: $\{x=g(r)\}$ between the two fluids. And $(U, V, P, \rho)$ is a week solution of \eqref{a3} in the distributional sense and $(U_i, V_i, P_i, \rho_i)$  solves the incompressible Euler system \eqref{a3} classically in $\O_i$ for
 $i=1,2$. (Please see Figure \ref{f11}).

\begin{figure}
\begin{minipage}[t]{0.5\linewidth}
\centering
\includegraphics[width=60mm,height=35mm]{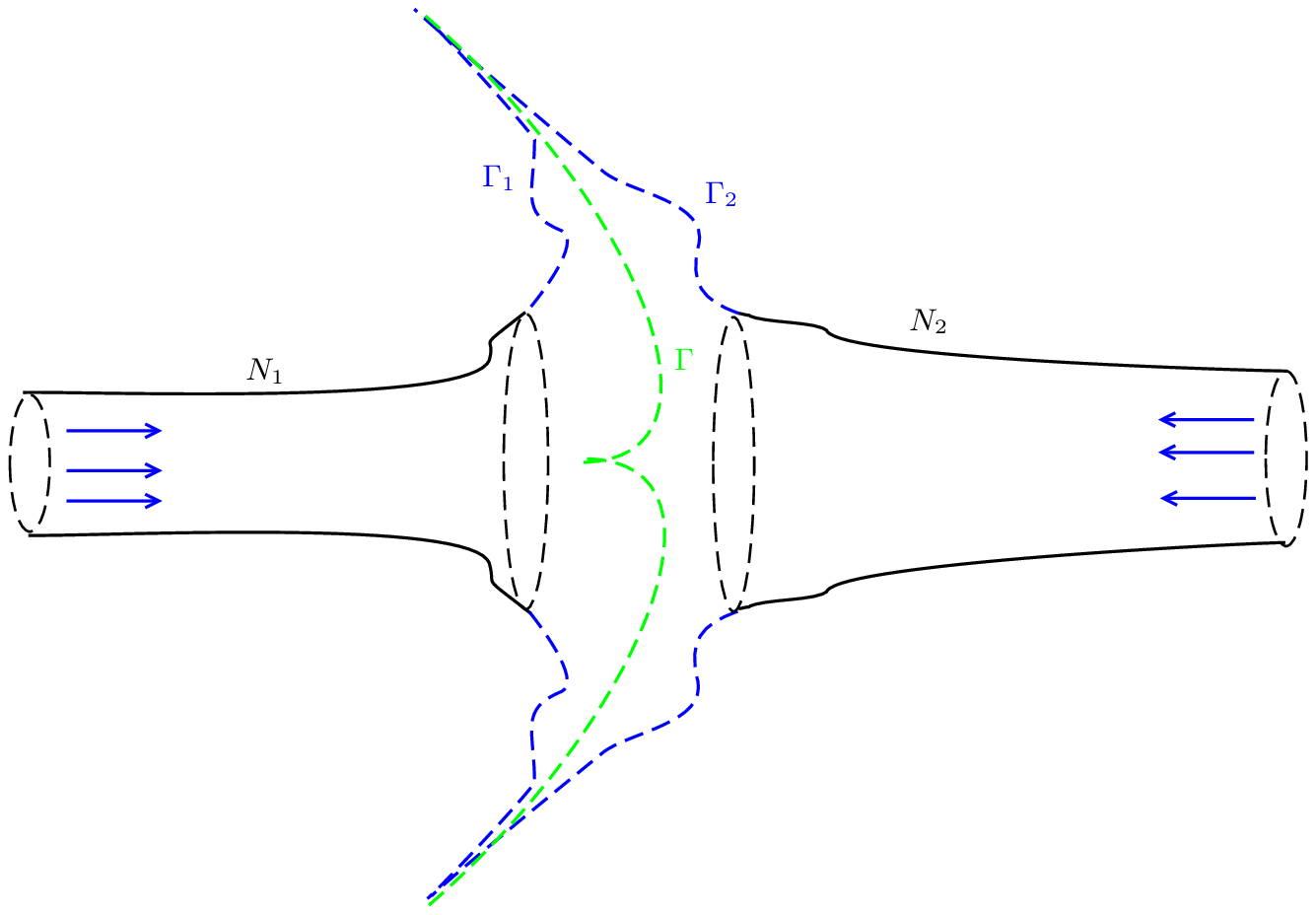}
\caption{Collision of two jets}\label{f11}
\end{minipage}%
\begin{minipage}[t]{0.5\linewidth}
\centering
\includegraphics[width=60mm,height=35mm]{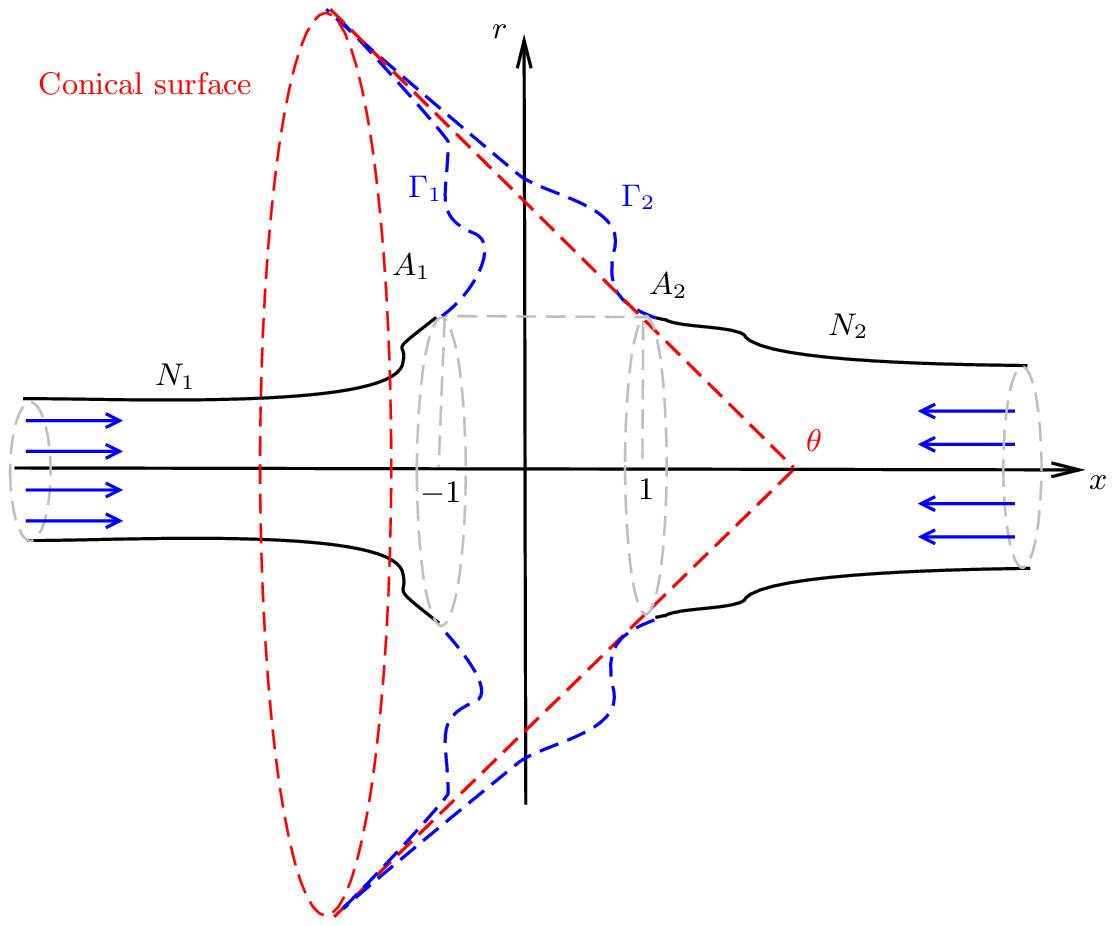}
\caption{Impinging outgoing jet}\label{f05}
\end{minipage}
\end{figure}

Then the Rankine-Hugoniot jump conditions
on $\Gamma$ become

 \be\label{a6}-\left [\ba{l}
\rho U\\
\rho U^2+P\\
 \rho UV\\
\ea\right]+g'(r)\left [\ba{l}
\rho V\\
\rho UV\\
 \rho V^2+P\\
\ea\right]=0,
 \ee
where $[\cdot]$ denotes the jump of a related function crossing the interface $\Gamma$.

Set $\mathfrak{m}_i=\rho_i\left(g'(r)V_i-U_i\right)$ ($i=1,2$) be
the mass flux across the interface, if
$\mathfrak{m}_1=\mathfrak{m}_2=0$ on the interface $\Gamma$, then
$(U, V, P, \rho)$ is a contact discontinuity. The Rankine-Hugoniot
conditions \eqref{a6} read as
\be\label{a7}-U_1+g'(r)V_1=0,\ \
-U_2+g'(r)V_2=0\  \text{and} \
 P_1=P_2.\ee

The condition \eqref{a7} implies that the
normal velocities on both sides of the interface $\Gamma$ vanish,
while the tangential velocity on both side of $\Gamma$ may have
nontrivial jump.

Furthermore, the well-known
Bernoulli's law can be written as $$
(U_i,V_i)\cdot\g\left(\frac{1}{2}(U_i^2+V_i^2)+\f{P_i}{\rho_i}\right)=0, \ \
\text{for}\ \ i=1,2,$$  namely,
$$\f{U_1^2+V_1^2}{2}+\f{P_1}{\rho_1}=\mathfrak{B}_1\ \ \text{and}\
\ \f{U_2^2+V_2^2}{2}+\f{P_2}{\rho_2}=\mathfrak{B}_2,$$
where $\mathfrak{B}_1$ and $\mathfrak{B}_2$ denote the Bernoulli's constants of the two fluids, respectively, in view of \eqref{a7}, then \be\label{a9}\rho_1\left(U_1^2+V_1^2\right)-\rho_2\left(U_2^2+V_2^2\right)=2\left(\rho_1\mathfrak{B}_1-\rho_2\mathfrak{B}_2\right)\triangleq \Lambda\ \ \ \text{on}\ \ \Gamma,\ee without loss of generality, we assume $\Lambda\geq0$.% However, the main purpose of this paper is to investigate the impinging outgoing jet with axis symmetric free surface $\Gamma_1$ and $\Gamma_2$, therefore, as the first step in this attempt, here we seek an impinging outgoing jet, which is $C^1$-smooth across the interface $\Gamma$.

%Here, we will seek an impinging outgoing jet with two free surfaces $\Gamma_1$ and $\Gamma_2$, which initial smoothly from the orifices of the nozzles $N_1$ and $N_2$, respectively, and extend to the infinity. Since the impinging outgoing jet approaches to some uniform flows in the far field, the free surfaces $\Gamma_1$ and $\Gamma_2$ have to get close enough to merge in far field. This is one main feature of the axially symmetric impinging jet here.

On another hand, on the free surfaces $\Gamma_1$ and $\Gamma_2$, the pressure is assumed to be the constant atmosphere pressure $P_{at}$ (in absence of gravity and surface tension), namely,\be\label{a09}P=P_{at}\ \ \text{on $\Gamma_1\cup\Gamma_2$.}\ee

Here is our problem of fluid mechanics: determine an impinging outgoing jet $(U, V, P, \rho)$ issuing from two nozzles $\mathcal{N}_1$ and $\mathcal{N}_2$ with two mass fluxes $M_1$ and $M_2$, bounded by two free surfaces $\Gamma_1$ and $\Gamma_2$ on which the pressure is a constant $P_{at}$. Furthermore, on the interface, the Rankine-Hugoniot conditions \eqref{a7} and \eqref{a9} hold.

On the solid walls $N_1$ and $N_2$, the flow satisfies the
slip-boundary condition, \be\label{a5} (U_i,V_i)\cdot \vec n_i=0,\ \
\ \ \ \text{on}\ \ N_i, \ee where $\vec n_i$ is the unit outer
normal of the wall $N_i$, for $i=1,2$. Moreover, on the symmetry axis $N_0$, \be\label{a05} V_i=0.\ee

Denote $M_1$ and $M_2$ as the mass fluxes in nozzles $\mathcal{N}_1$
and $\mathcal{N}_2$, respectively, then
\be\label{a10}\int_{\Sigma_i}(r U,r V,0)\cdot\vec{l_i}dS=\f{M_i}{2\pi},\ee where $\Sigma_i$ is any curve transversal to
the $x$-axis direction and $\vec{l_i}$ is the normal of $\Sigma_i$
in the positive $x$-axis direction for $i=1,2$.

\subsection{Impinging outgoing jet problem and main results}

We define the axially symmetric impinging outgoing jet problem as
follows.

 {\bf
Axially symmetric impinging outgoing jet problem.} For given any
mass fluxes $M_1>0$ and $M_2<0$ in the two semi-infinitely long
axially symmetric nozzles $\mathcal{N}_1$ and $\mathcal{N}_2$,
respectively, there exists an axially symmetric impinging outgoing
jet extending to the infinity, the free stream surfaces initiate at
the edges of the nozzles smoothly and shrink to some conical surface
at the far field, and a smooth interface separates the two jets,
furthermore, the pressure remains a constant on free stream surfaces
(see Figure \ref{f05}).

Next, we give the definition of the solution to the impinging
outgoing jet problem.

 {\bf A solution to the axially symmetric
impinging outgoing jet problem.}  A quintuple
$(U,V,P,\Gamma_1,\Gamma_2)$ is called a solution to the axially
symmetric impinging outgoing jet problem, provided that

(1). The smooth surfaces $\Gamma_1$ and $\Gamma_2$ are given by two
functions $x=g_1(r)\in C^1{((R,+\infty))}$ and $x=g_2(r)\in
C^1{((R,+\infty))}$ with $g_1(r)<g_2(r)$, and
\be\label{a12}g_1(R+0)=f_1(R-0),\quad g_2(R+0)=f_2(R-0)\quad
\text{({\it continuous fit conditions}),}\ee and
\be\label{a13}g_1'(R+0)=f_1'(R-0), \quad g_2'(R+0)=f_2'(R-0)
\quad\text{({\it smooth fit conditions})}.\ee Moreover, there exists
an asymptotic direction $\nu=(\cos\theta,\sin\theta)$ with $\th\in(0,\pi)$, such that
$g_1$ and $g_2$ are close to the asymptotic direction $\nu$ at far
field (See Figure \ref{f02}), ie., \be\label{a014}
\lim_{r\rightarrow
\infty}\left(g_2(r)-g_1(r)\right)=0\quad\text{and}\quad\lim_{r\rightarrow
\infty}g'_1(r)=\lim_{r\rightarrow\infty}g'_2(r)=\cot\theta,\ee the angle $\th$ is called the asymptotic deflection angle of the impinging outgoing jet.

(2). Denote the flow field $G$ bounded by the symmetric axis $N_0$,
the nozzle walls $N_1, N_2$ and the free boundaries $\Gamma_1,
\Gamma_2$. $(U,V,P)\in \left(C^{1,\alpha}(G)\cap
C^{0}(\overline{G})\right)^3$ solves the steady incompressible
Euler system \eqref{a3}, the boundary condition \eqref{a5}, the Rankine-Hugoniot conditions \eqref{a7} and the
mass flux conditions \eqref{a10};

(3). The radial velocity $V$ is positive in flow field and its
closure, except the symmetric axis and interface $\Gamma$, namely, $V>0$ in
$\bar G\setminus \left(N_0\cup\Gamma\right)$;

(4). $P=P_{at}$ on $\Gamma_1\cup\Gamma_2$;

(5). The interface $\Gamma$ satisfies the condition \eqref{a9}.

The first result in this paper is the existence of the impinging outgoing jet as follows.

\begin{theorem}\label{thm1}
For any given atmosphere pressure $P_{at}$, mass fluxes $M_1>0$, $M_2<0$ and $\Lambda\geq0$ issuing from the two
axially symmetric nozzles $\mathcal{N}_1$ and $\mathcal{N}_2$,
respectively, there exists a solution $(U,V,P,\Gamma_1,\Gamma_2)$
to the axially symmetric impinging outgoing jet problem. Furthermore, there
exists a $C^1$-smooth surface $x=g(r)$ satisfying $g_1(r)<g(r)<g_2(r)$ for $R<r<\infty$, which separates the two fluids and
initiates at the branching point on the symmetric axis (Figure
\ref{f04}). Furthermore, there exists a positive constant $\ld$, such that \be\label{a01} r
(g(r)-g_1(r))\rightarrow
\f{M_1}{2\pi\sqrt{\rho_1(\Lambda+\ld})\sin\th}\ \ \text{and}\ \ r(g(r)-g_2(r))\rightarrow
\f{M_2}{2\pi\sqrt{\rho_2\ld}\sin\th}\quad\quad\text{as}\quad
r\rightarrow+\infty. \ee
\end{theorem}

%\begin{figure}[!h]
%\includegraphics[width=110mm]{2AS-Impinging03.eps}
%\caption{Asymptotic direction of the impinging jet}\label{f03}
%\end{figure}

We would like to give the following comments on the existence
theorem.
%\begin{remark}
%In fact, the quantities $\sqrt{\f{\Lambda+\ld}{\rho_1}}$ and $\sqrt{\f{\ld}{\rho_2}}$ are nothing but the speed of the impinging outgoing jet for the fluid I and fluid II in downstream, respectively. And the positive parameter $\ld$ is unknown in advance, which is regarded as a parameter to solve the free boundary problem.
%\end{remark}

%\begin{remark}
%In this paper, we look for an impinging outgoing jet with positive radial velocity, and then the restriction of the asymptotic deflection angle $\th\in(0,\pi)$ is reasonable. Moreover, if the
%asymptotic direction approaches to the critical cases
%$(1,0)$ or $(-1,0)$, the one of the free
%surfaces has to disappear. The fact will be shown in Section \ref{ss}.
%\end{remark}

\begin{remark}One of key points in this work is the appearance of the interface between the two fluids, which is also a free boundary and is determined by the solution itself. In this paper, the impinging outgoing jets possess a smooth surface separating the two immiscible fluids, which intersects the symmetric axis at a unique point, called the \emph{branching point}. However, the appearance of the interface takes many essential difficulties to solve the free boundary problem in mathematics. The first one is the non-trivial jump of the velocity field on the interface (see \eqref{a9}), which leads that we have to seek a non-smooth solution in the whole fluid field. The second one is that the interface is common boundaries of two fluids, which is totally free. And we shall define the interface as the level set of the stream function and show that it is indeed a smooth curve. The third one is the regularity of the two-phase fluids near the branching point. \end{remark}

\begin{figure}
\begin{minipage}[t]{0.5\linewidth}
\centering
\includegraphics[width=60mm,height=35mm]{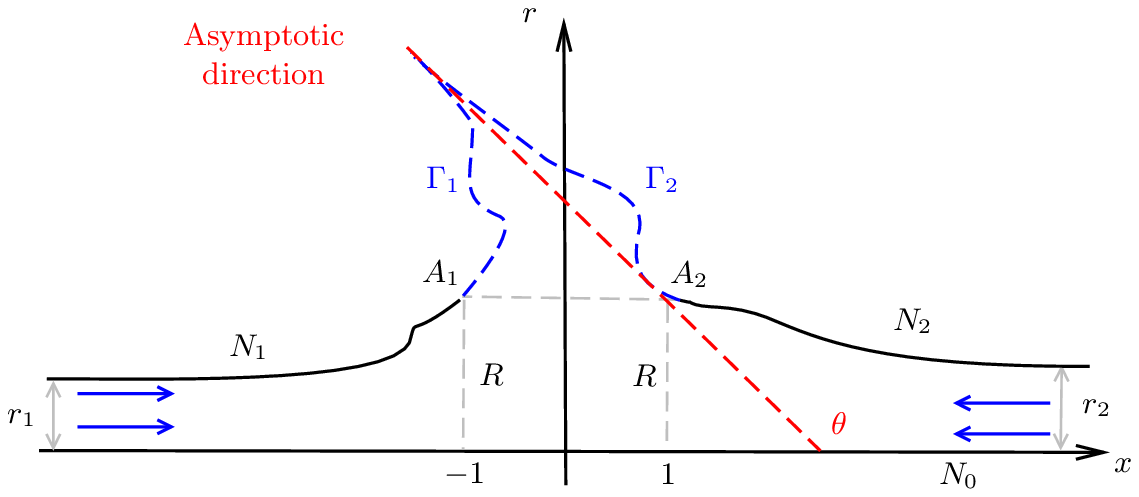}
\caption{Axisymmetric impinging outgoing jet flow in cylindrical
coordinates}\label{f02}
\end{minipage}%
\begin{minipage}[t]{0.5\linewidth}
\centering
\includegraphics[width=60mm,height=35mm]{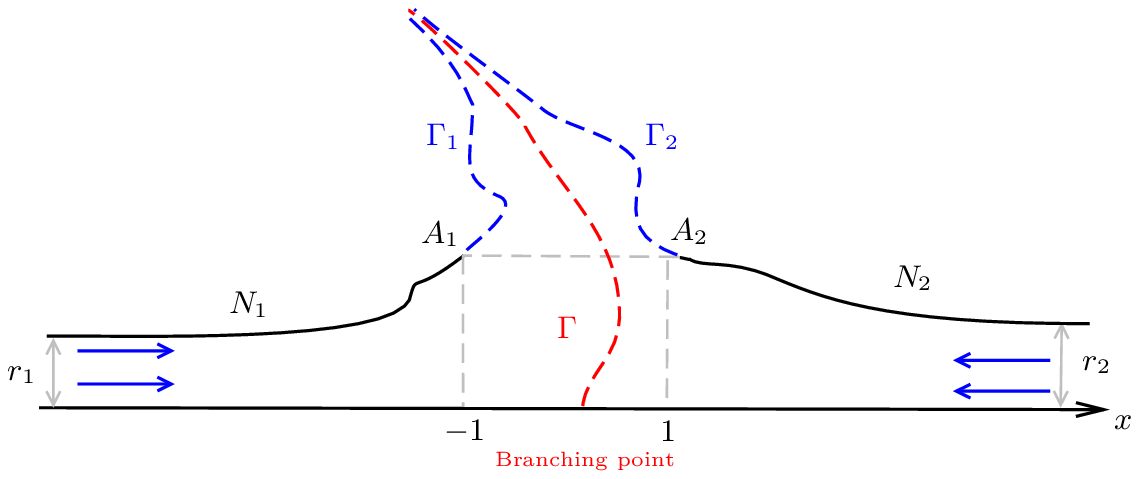}
\caption{Impinging outgoing jet and the interface
$\Gamma$}\label{f04}
\end{minipage}
\end{figure}

\begin{remark}There are many numerical results on the impinging free jets in absence of rigid nozzle walls, such as unsymmetrically impinging jets in \cite{KE}, impinging free jets in \cite{HW}, compressible impinging jets in \cite{CXF}. However, here we have to consider the geometry of both solid boundaries and free boundaries, the one of main difficulties is to verify the continuous fit and smooth fit conditions. In present work, an essential point is that we can choose a suitable pair of parameters $(\ld,\th)$, such that the continuous fit conditions are fulfilled. In other word, the parameters $\ld$ and $\th$ can be determined by the continuous fit conditions,
 which is the main difference from the analysis of impinging free jets without rigid boundaries. Therefore,
we first solve the free boundary problem for any $\ld$ and $\th$,
and then show the existence of a suitable pair of parameters $(\ld,
\th)$ to guarantee the continuous fit conditions of impinging
outgoing jet. Furthermore, we can show that the continuous fit
conditions imply the smooth fit conditions.\end{remark}

Theorem \ref{thm1} gives that there exists a pair of parameters
$(\ld,\th)$ to guarantee the existence of the axially symmetric
impinging outgoing jet. However, to the best of our knowledge, the
uniqueness of the jet with two free boundaries is totally open.
Next, for $\Lambda=0$, we give the uniqueness results on the axially
symmetric impinging outgoing jet, the idea borrows from the recent
work \cite{CDW} on the uniqueness of the asymmetric incompressible
jet.
% In fact, for $\Lambda=0$, in Theorem \ref{thm2}, we will show that the impinging outgoing jet with two free boundaries is unique, provided that the parameters $\ld$ and $\th$ are fixed.

\begin{theorem}\label{thm2}(Uniqueness of the axially impinging outgoing jet)(1) Given any parameters $(\ld, \th)$, such that the continuous fit conditions \eqref{a12} hold, then the axially symmetric impinging outgoing jet $(U,V,P,\Gamma_1, \Gamma_2)$ established in Theorem \ref{thm1} is unique.
\\ (2) Suppose that there exist two pairs of the parameters $(\ld,\th)$ and $(\ld, \tilde{\th})$, such that the continuous
fit conditions \eqref{a12} to the axially symmetric impinging outgoing jet hold, then
$\th=\tilde{\th}$.\end{theorem}

%\begin{remark}
%An important and interesting problem still remains open is whether
%the parameter $\lambda$ is unique under the present situation.
%\end{remark}

 Next, we give the asymptotic behaviors and the decay rate of the impinging
outgoing jet in the far field.

\begin{theorem}\label{thm4} The impinging outgoing jet flow $(U,V,P,\Gamma_1,\Gamma_2)$ established in Theorem \ref{thm1} satisfies the following asymptotic behavior in far fields,
 \be\label{a14}(U(x,r),V(x,r), P(x,r))\rightarrow
\left(\f{M_1}{\pi\rho_1 r_1^2},0, \f{\ld+\Lambda}{2\rho_1}+P_{at}-\f{M_1^2}{2\rho_1\pi^2r_1^4}\right),\ee
and\be\label{a15}\g U\rightarrow0, \ \ \g
V\rightarrow0,\ \ \g P\rightarrow0, \ee  as $x\rightarrow -\infty$, in any compact subset of $(0,r_1)$ and
\be\label{a16}(U(x,r),V(x,r), P(x,r))\rightarrow
\left(\f{M_2}{\pi\rho_2 r_1^2},0, \f{\ld}{2\rho_2}+P_{at}-\f{M_2^2}{2\rho_2\pi^2r_2^4}\right),\ee and
\be\label{a17}\g U\rightarrow0, \ \ \g
V\rightarrow0,\ \ \g P\rightarrow0, \ee as $x\rightarrow +\infty$, in any compact subset of $(0,r_2)$, and in the downstream,
\be\label{a07}(U(x,r),V(x,r), P(x,r))\rightarrow
\left(\sqrt{\f{\Lambda+\ld}{\rho_1}}\cos\th,\sqrt{\f{\Lambda+\ld}{\rho_1}}\sin\th,
P_{at}\right),\ee uniformly in any compact
subset of $\O_1$ as $r\rightarrow+\infty$, and
\be\label{a08}(U(x,r),V(x,r), P(x,r))\rightarrow
\left(\sqrt{\f{\ld}{\rho_2}}\cos\th,\sqrt{\f{\ld}{\rho_2}}\sin\th,
P_{at}\right),\ee  uniformly in any compact
subset of $\O_2$ as $r\rightarrow+\infty$, and
\be\label{a09}\g U\rightarrow0, \ \ \g V\rightarrow0, \ \ \g
P\rightarrow0 ,\ee uniformly in any compact subset of
$\O_1\cup\O_2$ as $r\rightarrow +\infty$.

Furthermore, for any $\alpha\in(0,2)$, one has \be\label{a09}r^\alpha\left(\left|U_1(x,r)-\sqrt{\f{\Lambda+\ld}{\rho_1}}\cos\th\right|+\left|V_1(x,r)-\sqrt{\f{\Lambda+\ld}{\rho_1}}\sin\th\right|\right)\rightarrow0,\ee \be\label{a009}r^\alpha\left(\left|U_2(x,r)-\sqrt{\f{\ld}{\rho_2}}\cos\th\right|+\left|V_2(x,r)-\sqrt{\f{\ld}{\rho_2}}\sin\th\right|\right)\rightarrow0,\ee as $r\rightarrow+\infty$.
%and %\be\label{a010}r\left(g_{2,\ld,\th}(r)-g_{1,\ld,\th}(r)\right)\rightarrow\f{M_1-M_2}{2\pi\ld\sin\th},\ee as $r\rightarrow+\infty$.
\end{theorem}

\begin{remark}The one of main differences between the impinging outgoing jet in two-dimensional and axially symmetric case is that the two-dimensional outgoing jet possesses a uniform positive width in far field, and however, the distance of free streamlines goes to zero in downstream in axially symmetric case. Here, we have to establish the decay estimates of outgoing jets and the free boundaries in far field. Indeed, the facts \eqref{a01} and \eqref{a09} give the decay rates of the velocity field and the distance of the two free stream surfaces in downstream. In particular, \eqref{a01} implies that the optimal decay rate of the distance of two free stream surfaces is $\f{1}{r}$ in downstream.
\end{remark}

\subsection{Motivation and history of the problem}
The motivation to investigate the impinging outgoing jets from two
nozzles comes from Chapter V. $\S$ 5 in the classical book \cite{BZ}
by G. Birkhoff and E. H. Zarantonello, in which the impinging
outgoing jets from two plane symmetric cylinders were considered.
Except for some simple channel geometries, the impinging problem of
two jets can not be solved analytically, as was shown in the
monographs \cite{BZ} and \cite{Gu}. Here, we consider the general
case that the impinging jet issuing from two axially symmetric
nozzles with variable cross-section.

\begin{figure}[!h]
\includegraphics[width=100mm]{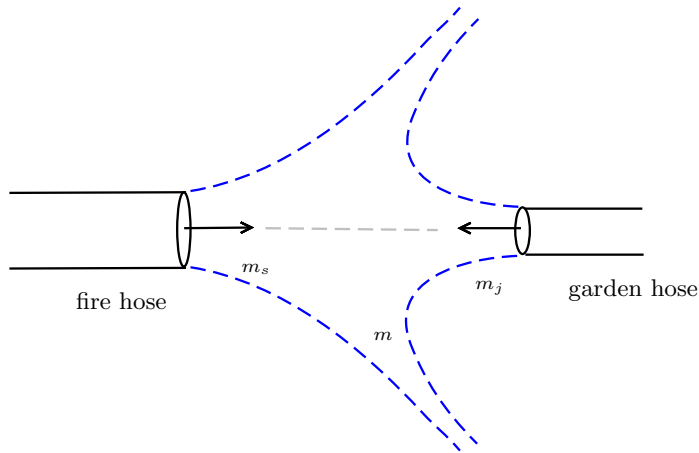}
\caption{Collision of two jets (Figure 16.7 in \cite{FA3})}\label{fi0}
\end{figure}

Another motivation to investigate the impinging outgoing jets
issuing from two nozzles comes from the Chapter 16 in famous survey
\cite{FA3}. The physical problem is also related to the shaped
charge question in \cite{CH}. As mentioned in Page 152 \cite{FA3},
"... \emph{we can formulate this problem as a collision of two jets,
say a garden hose and a fire hose; see Figure 16.7.}" (see Figure
\ref{fi0}) and A. Friedman proposed an open problem that

``\emph{Problem (1). Analyze the axially symmetric free boundary
problem associated with the flow in Figure 16.7 in incompressible
case.}"

On another side, there are many numerical results on this impinging
outgoing jet problem, such as the incompressible problem for an
arbitrary polygonal nozzle in \cite{DET}, and the incompressible jet
with gravity in \cite{DVB} and so on. Moreover, Hurean and Weber in
\cite{HW} considered the impinging of two incompressible ideal free
jets (in the absence of rigid nozzle walls) numerically, and some
existence results on two compressible free jets were also
investigated in \cite{CXF}. However, we also would like to mention
the numerical result on asymmetric impinging free jets in \cite{KE}.

The study of liquid jets issuing from containers is centuries old. As far back as 1868, Helmholtz and Kirchhoff introduced the
classical theory of free streamlines in two-dimensional jets. The
steady irrotational flows of ideal incompressible fluid, bounded by
nozzle walls and free streamlines were investigated. The following
decades saw extensions of a great many different kinds of
two-dimensional flows, on the basis of the complex analysis methods
by Planck, Joukowsky, R\'ethy, Levi-Civita, Greenhill and others.

Some substantial post-war monographs are those of
Birkhoff-Zarantonello \cite{BZ}, Gurevich \cite{Gu}, Milne-Thomson
\cite{MT}. For two-dimensional irrotational case, a generalized
Schwarz-Christoffel transformation, combined with a Fourier
technique to formulate a free boundary problem into a nonlinear
integral-differential equation, some existence results on jets in
special nozzles have been constructed. However, two-dimensional jets
have been given most of the attention in the existence theory, and
the limited amount of work on axisymmetric jets has been confined.
The reason is that the complex analysis method which has been
adapted to two-dimensional jets has noneffective in the axially
symmetric case. A first breakthrough work on the axially symmetric
free streamline was due to Garabedian, Lewy and Schiffer in
\cite{GLS} in 1952, in which some existence results on the axially
symmetric cavity were established by variational approach.
Furthermore, Alt, Caffarelli and Friedman developed the variational
method to deal the free streamlines problem in their elegant works
\cite{AC1,ACF7}. Based on their framework, some well-posedness
results on axially symmetric jet in \cite{ACF3}, asymmetric jet in
\cite{ACF1}, jet with gravity in \cite{ACF2}, and jets with two
fluids in \cite{ACF5, ACF6} have been established. In this paper,
some fundamental ideas on the existence theory are still borrowed
from the variational method in \cite{AC1}. Recently, if we assume
that the fluid is smooth across the interface $\Gamma$ (ie. $\Ld=0$)
apriorily, some existence results for incompressible plane symmetric
impinging outgoing jets has been obtained in \cite{DW}. In fact, the
interface $\Gamma$ is a contact discontinuity and the jump $\Ld$ is
always non-trivial and non-zero generally. As we mentioned before,
we have to investigate the non-zero jump $\Ld\neq 0$, which is one
of main differences between this paper to the previous paper
\cite{DW}. As a first step to attack the original problem on
impinging outgoing jets with nontrivial jump, Wang and Xiang in
\cite{WX} considered a toy model on the incompressible fluids
issuing from two infinity long co-axis and symmetric nozzles without
jet free boundary and established some properties on the contact
discontinuity between the two fluids. Therefore, the objective of
the present paper is to establish the well-posedness theory on the
impinging outgoing jet problem and solve the open problem proposed
by A. Frideman in 1989.

\subsection{Methodology}
From the physical problem here, the main difference and difficulty
here stems from the shape of nozzle walls, we have to find a
mechanism, such that the free boundaries of the jets connecting
smoothly at the edge of the nozzle walls (so-called \emph{continuous
fit and smooth fit conditions}). Another main difficulty is how to
analyze the interface with contact discontinuity between the two
incompressible ideal fluids.

We would like to comment the main ideas of the proof as follows. The
variational method developed by H. Alt, L. Caffarelli and A.
Friedman in 1980s has been shown to be powerful and effective to
solve the free streamline theory for more general models. In
two-dimensional case, the stream function is harmonic in the fluid
domain, while in axially symmetric case, it solves a linear elliptic
equation with some lower-order term, and we have to deal with the
regularity near the symmetric axis. This is the first difficulty in
this paper. The second one is that the distance of the two free
boundaries converges to a positive constant in two-dimensional case (see \cite{DW}),
and however, the distance goes to zero in axially symmetric jet
here. Therefore, we can not borrow some uniform special flow to show
some fundamental properties of the free boundaries as in 2D case,
such as the vanishing and non-vanishing of free boundary, and the
asymptotic behaviors of the jet in far fields, and so on. Here, our
strategy is to establish firstly the decay rate of distance of the
two free boundaries and the impinging outgoing jet in far field, and
then to obtain the desired properties via some rescaling arguments.
The third principal difficulty in this paper centers about how to
choose the suitable parameters $(\lambda, \th)$ to assure the
continuous fit conditions in impinging outgoing jets. Some
continuous dependent relationships and monotonic properties to the
impinging outgoing jets with respect to the parameters $(\ld, \th)$
are established and guarantee the fact to be fulfilled. The fourth difficulty here is the occurrence of three free boundaries $\Gamma_1$, $\Gamma_2$ and $\Gamma$, which is main difference to the previous results, such as \cite{DW,WX}. In particular, the solution is not smooth across the interface $\Gamma$ and it is a contact discontinuity between the two fluids. To our knowledge, this is the first well-posedness work on the jet flows problem with three free boundaries.
%Moreover, the proof of the asymptotic
%behavior of the impinging outgoing jet is based on the blowup
%argument, which has been used to deal with the subsonic compressible
%flows in infinitely long nozzles in \cite{CDSW, CDX, CHW, CHWX, DD,
%DWX, DX, DXX,DXY, DL, XX1,XX2,XX3}.

The remain of this paper is organized as follows. First, we
establish the free boundary value problem to the physical problem in
Section 2. The solvability of the free boundary value problem
follows from the standard variational approach, which has been
developed by Alt, Caffarelli and Friedman in the celebrated works
\cite{AC1, ACF1, ACF3}. Moreover, some properties of the free
boundaries will be obtained and we can verify the continuous fit and
smooth fit conditions for suitable parameters $\ld$ and $\th$.
Additionally, we will investigate the existence and properties of
the interface between the two fluids. Hence, we can obtain the
existence of impinging outgoing jet in Section 3. In Section 4, we
will give the uniqueness of the impinging outgoing jet and the
parameters. In Section 5, the asymptotic behavior of the impinging
outgoing jet is obtained along the blow-up argument, which has been
used to deal with the subsonic compressible flows in infinitely long
nozzles in \cite{DD, DWX, DX, DXX,DXY, DL, XX1,XX2,XX3}. Some results on the variational problem are given in the Appendix.
\section{Mathematical settings of the free boundary problem}
Due to the continuity equation \eqref{a3}, one can introduce stream
functions $\Psi_i(x,r)$ ($i=1,2$) such that \be\label{b1}
\p_x\Psi_i=- r\rho_iV_i,\ \ \ \ \p_r\Psi_i= r\rho_iU_i. \ee

In order to convenient to the analysis, we introduce the scaled stream function
$\psi_i=\f{\Psi_i}{\sqrt{\rho_i}}$ ($i=1,2$), and denote $\psi$ as
$$\psi=\left\{ \ba{l}
\psi_1\ \ \ \text{in}\ \ \O_1,\\
\psi_2\ \ \ \text{in}\ \ \O_2.
 \ea \right.$$

This together with the irrotational condition gives \be\label{b01}\Delta\psi-\f{1}{r}\f{\p\psi}{\p r}=0\ \
\text{in $\O_1\cup\O_2$}.\ee Here and after, $\O_i$ denotes the flow field,
bounded by the nozzle walls $N_i$, the symmetric axis $N_0$, the interface $\Gamma$ and the free
boundaries $\Gamma_i$ ($i=1,2$).

In this paper, we expect to seek an axially symmetric impinging outgoing jet flow with
positive vertical velocity, and thus denote $\O$ bounded by $N_i$,
$N_0$ and $L_i$ ($i=1, 2$) as the possible flow field of impinging outgoing jet (see Figure \ref{fi2}), where
$$L_1=\{(x, r)\mid r=R, x<-1\}\ \ \text{and}\ \ L_2=\{(x, r)\mid r=R,
x>1\}.$$

\begin{figure}[!h]
\includegraphics[width=100mm,height=40mm]{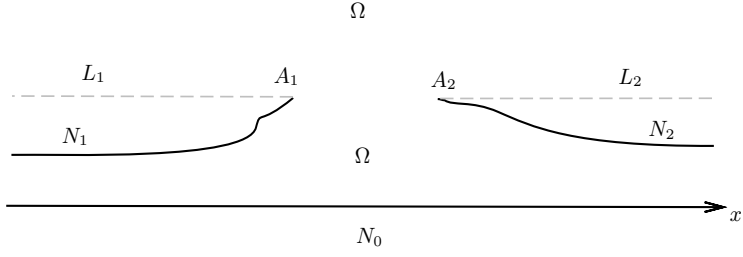}
\caption{The possible flow field $\O$}\label{fi2}
\end{figure}

Moreover, the nozzles $N_1\cup N_2$ and the free boundaries
$\Gamma_1\cup\Gamma_2$ are streamlines, thus $\psi$ remains some
constant on those boundaries, without loss of generality, we can
impose $\psi=m_1$ on $N_1\cup \Gamma_1$ and $\psi=m_2$ on
$N_2\cup\Gamma_2$. Thus, the free boundaries $\Gamma_1$ and
$\Gamma_2$ can be defined as
$$\Gamma_1= \O\cap\p\left\{\psi<m_1\right\},\ \ \Gamma_2=
\O\cap\p\left\{\psi>m_2\right\},$$ respectively, where $\Omega$ is
the possible flow field defined before and
$m_i=\f{M_i}{2\pi\sqrt{\rho_i}}$ ($i=1,2$). And the constant
pressure boundary condition on free boundaries can be rewritten as
\be\label{b2}\left|\f
1r\f{\p\psi}{\p\nu}\right|=\sqrt{\Lambda+\lambda} \ \
\text{on}~~\Gamma_1,\ \ \ \ \  \left|\f
1r\f{\p\psi}{\p\nu}\right|=\sqrt{\lambda} \ \
\text{on}~~\Gamma_2,\ee where $\nu$ is unit outward normal to the
free stream surfaces $\Gamma_1$ and $\Gamma_2$.

Hence, we formulate the boundary value problem to the stream function
$\psi$,
\be\label{b3}\left\{\ba{ll} \Delta\psi-\f1r\f{\p\psi}{\p r}=0,\ \ \ &\text{in}\ \ \ \O_1\cup\O_2,\\
\left|\f 1r\f{\p\psi}{\p\nu}\right|=\sqrt{\Lambda+\lambda},\ \ \ \
&\text{on}\ \ \ \Gamma_1,\ \ \ \ \  \left|\f 1r\f{\p\psi}{\p\nu}\right|=\sqrt{\lambda},\ \ \ \
\text{on}\ \ \ \Gamma_2,\\ \left|\f{\g\psi^+}{r}\right|^2-\left|\f{\g\psi^-}{r}\right|^2=\Lambda,\ \ &\text{on}\ \ \ \Gamma,\\ \psi=m_1,\ \ \ &\text{on}\ \ \
N_1\cup\Gamma_1,\ \ \ \ \  \psi=m_2,\ \ \ \text{on}\ \
N_2\cup\Gamma_2,\\
\psi=0,\ \ \ &\text{on}\ \ N_0\cup\Gamma,\ea\right. \ee where $\psi^\pm(X_0)$
($X_0\in \Gamma$) denotes the limit of $\psi(X)$ with
$X\in\{\pm\psi>0\}$, as $X\rightarrow X_0$.

We would like to emphasize that the undetermined constants $\lambda$
and $\theta$ are regarded as two parameters to solve the free
boundary problem. We will solve the free boundary problem for any
$\lambda$ and $\theta$, and then to show the existence of suitable
parameters to guarantee the continuous fit conditions.

\section{Existence of the impinging outgoing jets}
\subsection{Truncated variational problem}
In order to solve the free boundary value problem \eqref{b3}, we
first introduce some notations and two auxiliary functions as
follows. Define domain $D$ as
$$D=\O\cap\left\{r<R\right\}.$$

%\begin{figure}[!h]
%\includegraphics[width=90mm]{2Impinging06.eps}
%\caption{Definitions of $D_1$ and $D_2$}\label{fi6}
%\end{figure}

Next, we define two bounded functions $\Phi_1$ and $\Phi_2$ as
follows,
$$\text{$\Delta\Phi_1-\f1r\f{\p\Phi_1}{\p r}=0$ in $D$, and $m_2<\Phi_1<m_1$ in $D$},$$ \be\label{b4}\begin{aligned}\Phi_1(x,r)=\left\{\ba{ll}
0,\ \ \ \ \ \ &\text{if}\ \
(x, r) \in N_0,\\
m_2,\ \ \ \ \ \ &\text{if}\ \ r_2<r\leq R,~ (x, r)~ \text{lies
right}~ N_2,\\ m_1,\ \ \ \ \ \ &\text{if}\ \ r_1<r\leq R,~ (x, r)~
\text{lies left}~ N_1,\\ m_1,\ \ \ \ \ \ &\text{if}\ \  (x,
r)\in\Omega\cap \left\{r\geq R\right\}, \ea\right.\end{aligned}\ee
and
$$\text{$\Delta\Phi_2-\f1r\f{\p\Phi_2}{\p r}=0$ in $D$, $m_2<\Phi_2<m_1$
in $D$},$$
\be\label{b5}\begin{aligned}\Phi_2(x,r)=\left\{\ba{ll} 0,\ \ \ \ \ \
&\text{if}\ \
(x, r)\in N_0,\\
m_2,\ \ \ \ \ \ &\text{if}\ \ r_2<r\leq R,~ (x, r)~ \text{lies
right}~ N_2,\\ m_1,\ \ \ \ \ \ &\text{if}\ \ r_1<r\leq R,~ (x, r)~
\text{lies left}~ N_1,\\ m_2,\ \ \ \ \ \ &\text{if}\ \  (x,
r)\in\Omega\cap \left\{r\geq R\right\}. \ea\right.\end{aligned}\ee

We introduce the admissible set as
$$K=\left\{\psi\in
H_{loc}^1(\O)|~ \Phi_2\leq\psi\leq\Phi_1\right\},$$ set $e=\left(-\sin\th,\cos\th\right)$ with $\th\in[0,\pi]$, and a functional \be\label{b6}J_{\ld,\th}(\psi)=\int_{\O}r\left|\f{\g\psi}{r}-\left(\sqrt{\Lambda+\ld} \chi_{\{0<\psi<m_1\}}+\sqrt{\ld} \chi_{\{m_2<\psi\leq0\}}\right)e\right|^2dxdr.\ee

Since the functional $J_{\ld,\th}$ is unbounded for any $\psi\in K$, we have to truncate the possible
flow field $\O$ and formulate truncated problems as
follows.

For any $\mu>1$ and $i=1, 2$, we define
\be\label{b006}\begin{aligned}&r_{i,
\mu}=\min\left\{r\right|(-1)^i\mu=f_i(r)\}, \quad \ H_{i,
\mu}=\left\{\left((-1)^i\mu, r\right)|~ 0\leq r\leq
r_{i,\mu}\right\},\\& N_{0,\mu}=N_0\cap\{-\mu<x<\mu\},\ \text{and}\
N_{i,\mu}=\left\{(x, r)| ~x=f_i(r), r_{i, \mu}< r\leq
R\right\}.\end{aligned}\ee

Moreover, we introduce a cutoff domain $\O_\mu$ (see Figure
\ref{f017}) as \be\label{b007}\text{$\O_{\mu}$ is
bounded by $N_{i, \mu}$, $L_i$, $N_{0,\mu}$ and $H_{i, \mu}$},\ee and denote $$D_{\mu}=\O_\mu\cap\left\{r<R\right\}.$$
\begin{figure}[!h]
\includegraphics[width=100mm,height=40mm]{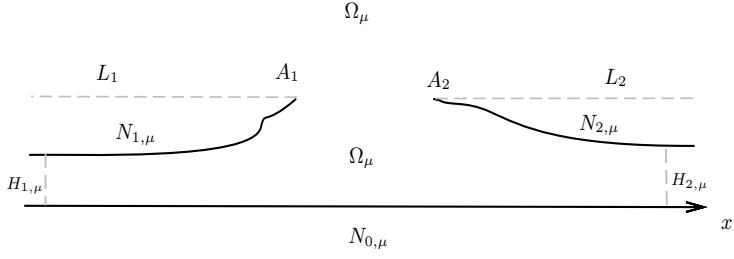}
\caption{The truncated domain $\O_\mu$}\label{f017}
\end{figure}

We also introduce an admissible set
$$K_{\mu}=\left\{\psi\in K\left|\psi=\f{m_1}{r_{1,\mu}^2}r^2~\text{on}~H_{1, \mu}, \right.~~\psi=\f{m_2}{r_{2,\mu}^2}r^2~\text{on}~H_{2, \mu}\right\},$$ and an auxiliary functional
$$J_{\lambda, \th,
\mu}(\psi)=\int_{\O_\mu}r\left|\f{\g\psi}{r}-\left(\sqrt{\Lambda+\ld} \chi_{\{0<\psi<m_1\}}+\sqrt{\ld} \chi_{\{m_2<\psi\leq0\}}\right)e\right|^2dX,\ \ \psi\in K_\mu,$$
where $\chi_E$ is the indicator function of the set $E$. Here and after, we denote $dX=dxdr$
and $X=(x,r)$ for simplicity.

{\bf The truncated variational problem ($P_{\ld,\th,\mu}$):} Find a $\psi_{\ld, \th,
\mu}\in K_\mu$, such that \be\label{b06}J_{\lambda, \th,
\mu}(\psi_{\ld, \th,\mu})=\mathop {\min }\limits_{\psi \in
K_\mu}J_{\lambda, \th, \mu}(\psi).\ee

Furthermore, the free boundaries of the truncated variational
problem ($P_{\ld,\th,\mu}$) are defined as follows.
\begin{definition}\label{bb0}The set $$\Gamma_{1, \mu}= \O_\mu\cap\p\left\{\psi_{\ld,\th,\mu}<m_1\right\},$$
is called the left free boundary, and $$\Gamma_{2,
\mu}=\O_\mu\cap\p\left\{\psi_{\ld,\th,\mu}>m_2\right\},$$ is called
the right free boundary.

Furthermore, define $$\Gamma_\mu=\O_\mu\cap\{\psi_{\ld,\th,\mu}=0\},$$ be the interface separating the two fluids.\end{definition}

\subsection{Existence of minimizer to the truncated variational problem} First, we give
the existence of the minimizer to the truncated variational problem.
\begin{prop}\label{bb0} For any $\ld>0$, $\th\in [0,\pi]$ and $\mu>1$, there exists a minimizer $\psi_{\ld, \th,\mu}\in K_\mu$ to the truncated variational problem ($P_{\ld,\th,\mu}$).
\end{prop}
\begin{proof}
Due to Theorem 1.3 in \cite{AC1}, it suffices to construct a function $\psi_0\in K_\mu$ such that $J_{\ld,\th,\mu}(\psi_0)<+\infty$.

{Case 1:} For $\th\in(0,\pi)$ in $\O_\mu$.

Indeed,  for some sufficiently large $$R_0>\max\left\{1, \f{m_1}{\sqrt{\Lambda+\ld}
(R+1)\sin\th}-(R+1)\cot\th, -\f{m_2}{\sqrt{\ld}
(R+1)\sin\th}+(R+1)\cot\th\right\},$$ and taking
$\overline{\psi}$ be a smooth function such that $\psi_0\in K_\mu$.
Define $\psi_0$ in $\O_\mu$ as follows
$$\psi_0(X)=\left\{\ba{ll} \min\left\{\max\left\{\sqrt{\Lambda+\ld} r\left(r\cos\th-x\sin\th\right),0\right\},m_1\right\},\ \ \ \ \ \ &\text{if}\ \
r\geq R+1,\ \ r\cos\th-x\sin\th\geq0,\\ \max\left\{\min\left\{\sqrt{\ld} r\left(r\cos\th-x\sin\th\right),0\right\},m_2\right\},\ \ \ \ \ \ &\text{if}\ \
r\geq R+1,\ \ r\cos\th-x\sin\th\leq0, \\
m_1,\ \ \ \ \ \ &\text{if}\ \ x\leq-R_0,\ \ R\leq r\leq R+1,\\
m_2,\ \ \ \ \ \ &\text{if}\ \ x\geq R_0,\ \ R\leq r\leq R+1,\\
\overline{\psi}(X),\ \ \ \ \ \ &\text{if $(x,r)\in\O_{\mu, R_0}$},\\
\eta(x)\f{m_1}{r_{1,\mu}^2}r^2+(1-\eta(x))\f{m_2}{r_{2,\mu}^2}r^2,\
\ \ \ \ \ &\text{if $(x,r)\in \O'_\mu$}.\ea\right.$$ Here, $\O'_\mu=\O_{\mu}\cap\{r\leq\min\{r_{1,\mu},r_{2,\mu}\}\}$, $\eta(x)$ be a cut-off function
satisfying
\be\label{b0b}\eta(x)=1\ \ \text{for}\ \
x\leq-\mu,\ \
\eta(x)=\f{\mu-x}{2\mu}\ \ \text{for}\ \ -\mu\leq x\leq \mu,\ \ \eta(x)=0 \ \ \text{for}\ \ x\geq \mu,\ee and
$$\O_{\mu, R_0}=\O_\mu\cap\left\{\min\{r_{1,\mu}, r_{2,\mu}\}\leq
r\leq R\right\}\cup\left\{-R_0\leq x\leq R_0, R\leq r\leq
R+1\right\}.$$

It is easy to check that $J_{\ld, \th,
\mu}(\psi_0)<+\infty$.

{\bf Case 2:} For $\th=0$ or $\pi$.

Without loss of generality, assume $\th=0$. It suffice to define a function $\psi_0(X)$ as follows. Set
$$\t\O_{\mu, R_0}=\O_\mu\cap\left\{\min\{r_{1,\mu},r_{2,\mu}\}\leq r\leq R\right\}\cup\left\{-2\leq x\leq 2, R\leq r\leq
R_0\right\},$$ for some sufficiently large
$R_0>\sqrt{R^2+\f{2m_1}{\sqrt{\Lambda+\ld}}-\f{2m_2}{\sqrt{\ld}}}$, and define a function
$\psi_0$ as
$$\psi_0(X)=\left\{\ba{ll}\f{\sqrt{\ld}(r^2-R^2)}{2}+m_2,\ \ \ \ \ \ &\text{if}\ \
x\geq2,\ \  R\leq r\leq \sqrt{R^2-\f{2m_2}{\sqrt{\ld}}},\\
\min\left\{\f{\sqrt{\Lambda+\ld}}2\left(r^2+\f{2m_2}{\sqrt{\ld}}-R^2\right),m_1\right\},\
\ \ \ \ \ &\text{if}\ \
x\geq2,\ \  r\geq \sqrt{R^2-\f{2m_2}{\sqrt{\ld}}},\\
m_1,\ \ \ \ \ \ &\text{if}\ \ x\leq2,\ \ r\geq R_0,\\ m_1,\ \ \ \ \ \ &\text{if}\ \ x\leq -2,\ \ R\leq r\leq R_0,\ \ \\
\overline{\psi}(X),\ \ \ \ \ \ &\text{if}\ \ (x, r)\in \t\O_{\mu,
R_0},\\
\eta(x)\f{m_1}{r_{1,\mu}^2}r^2+(1-\eta(x))\f{m_2}{r_{2,\mu}^2}r^2,\
\ \ \ \ \ &\text{if $(x,r)\in\O_\mu\cap\left\{r\leq\{r_{1,\mu},
r_{2,\mu}\}\right\}$} ,\ea\right.$$ where $\eta(x)$ is defined as
\eqref{b0b}, and $\overline{\psi}$ be a smooth function
such that $\psi_0\in K_\mu$.

Therefore, we finish the proof of Proposition \ref{bb0}.
\end{proof}

Next, we will obtain the regularity of the minimizer.
\begin{prop}\label{cc8} Let $\psi_{\ld, \th, \mu}$ be a minimizer to the truncated variational problem
($P_{\ld,\th,\mu}$), and for any open subset
$\O_0\subset\subset\O_\mu\cap\{m_2<\psi_{\ld,\th,\mu}<m_1\}\cap\{\psi_{\ld,\th,\mu}\neq0\}$, then
$\psi_{\ld, \th,\mu}\in C^{0, 1}(\O_{\mu})$, $\psi_{\ld,\th,\mu}\in
C^{2,\sigma}(\O_0)$ and $\psi_{\ld, \th, \mu}\in C^{1,
\sigma}(\O_0\cup N_{1,\mu}\cup N_{2,\mu})$ for some $0<\sigma<1$.
\end{prop}
\begin{proof}Firstly, $\psi_{\ld, \th,\mu}\in C^{0,
1}(\O_\mu)$ follows in the same manner as Corollary 4.4 in \cite{ACF5}.

Next, the standard interior estimates to linear elliptic equation in
Chapter 8 in \cite{GT} gives $\psi_{\ld,\th,\mu}\in C^{2,\sigma}(\O_0)$
and $\psi_{\ld,\th, \mu}\in C^{1, \sigma}(N_{1,\mu}\cup N_{2,\mu})$.

The regularity of $\psi_{\ld, \th, \mu}$ near the axis $N_{0,\mu}$ can be obtained by the standard arguments as in \cite{DD} and \cite{XX2}.

Therefore, we finish the proof of Lemma \ref{cc8}.
\end{proof}

\subsection{Uniqueness and monotonicity of the minimizer}

Firstly, we will give a lower bound and an upper bound to the
minimizer $\psi_{\ld, \th, \mu}$.
\begin{lemma}\label{cc9} For any minimizer $\psi_{\ld, \th, \mu}$ to the variational problem
($P_{\ld,\th,\mu}$), one has \be\label{b18}\max\left\{\f{m_2}{r_{2,\mu}^2}r^2,
m_2\right\}\leq\psi_{\lambda, \th, \mu}(x,
r)\leq\min\left\{\f{m_1}{r_{1,\mu}^2}r^2,m_1\right\}\ \ \ \
\text{in}\ \ \ \O_{\mu},\ee where $r_{1,\mu}$ and $r_{2,\mu}$ are defined
in \eqref{b006}.
\end{lemma}
\begin{proof}Firstly, consider the lower bound of $\psi_{\ld,\th,\mu}$.

Set $\phi_1=\max\left\{\f{m_2}{r_{2,\mu}^2}r^2, m_2\right\}$,
$\phi_2=\min\left\{\f{m_1}{r_{1,\mu}^2}r^2,m_1\right\}$ and
$\psi=\psi_{\ld,\th,\mu}$ for simplicity.

Firstly, since $\psi\in K_{\mu}$, one has
\be\label{bb}m_2\leq\psi\leq m_1\ \ \text{in}\ \ \O_{\mu}.\ee

Next, we shall prove that $$\phi_1\leq\psi\ \ \text{in $\O_\mu$}.$$
Due to the fact that $\max\left\{\psi, \phi_1\right\}\in K_{\mu}$,
we have
$$J_{\ld,\th,\mu}(\psi)\leq J_{\ld,\th,\mu}(\max\left\{\psi, \phi_1\right\}).$$
Furthermore, the fact \be\label{b012}\phi_1=m_2\ \ \text{for $r\geq
r_{2,\mu}$},\ee gives that $$\psi\geq\phi_1\ \ \text{in}\ \ \O_{\mu}\cap\{r\geq r_{2,\mu}\}.$$

Therefore, we obtain
\be\label{b12}\begin{aligned}0\geq&\int_{\O_{\mu,r_{2,\mu}}}r\left|\f{\g\psi}{r}-\left(\sqrt{\Lambda+\ld}\chi_{\{0<\psi<m_1\}}
+\sqrt{\ld}\chi_{\{m_2<\psi\leq0\}}\right)\cdot
e\right|^2dX\\&-\int_{\O_\mu,r_{2,\mu}}r\left|\f{\g\max\left\{\psi,
\phi_1\right\}}{r}-\left(\sqrt{\Lambda+\ld}\chi_{\{0<\max\left\{\psi,
\phi_1\right\}<m_1\}} +\sqrt{\ld}\chi_{\{m_2<\max\left\{\psi,
\phi_1\right\}\leq0\}}\right)\cdot
e\right|^2dX,\end{aligned}\ee here, $\O_{\mu,r_{2,\mu}}=\O_\mu\cap\{r<r_{2,\mu}\}$. This together with the similar arguments as Lemma 3.4 in \cite{WX} yields to
$$\int_{\O_{\mu, r_{2,\mu}}}\left|\g\min\left(\psi-\phi_1,
0\right)\right|^2dX\leq0,$$ which implies that
$$\psi-\phi_1=\text{constant} \ \ \ \ \text{in} ~\O_{\mu, r_{2,\mu}}.$$
Since $\psi\geq\phi_1$ on $\p\O_{\mu, r_{2,\mu}}$, we conclude that
$$\psi\geq\phi_1\ \ \ \ \text{in}~\O_{\mu, r_{2,\mu}}.$$

Therefore, we obtain the lower bound of $\psi_{\ld,\th,\mu}$ in
\eqref{b18}.

Next, we can now proceed as before to show the upper bound of $\psi_{\ld,\th,\mu}$.

Similarly, it suffices to prove that the upper bound holds in
$\O_{\mu,r_{1,\mu}}=\O_\mu\cap\left\{r<r_{1,\mu}\right\}$. Due to \eqref{b07}, one has
$$\Delta\psi_{\ld,\th,\mu}-\f{1}{r}\f{\p\psi_{\ld,\th,\mu}}{\p
r}\geq0\ \ \text{in $\O_{\mu,r_{1,\mu}}$ in a weak sense,}$$ and $\psi_{\ld, \th, \mu}\leq
\f{m_1}{r_{1, \mu}^2}r^2$ on $\p\O_{\mu,r_{1,\mu}}$, then the maximum principle
implies $\psi_{\ld, \th, \mu}(x,r)\leq \f{m_1}{r_{1,\mu}^2}r^2$ in
$\O_{\mu,r_{1,\mu}}$.
 We complete the proof of Lemma \ref{cc9}.
\end{proof}

Next, in view of Lemma \ref{cc9}, using the similar
arguments Proposition 3.5 in \cite{WX}, we will establish the uniqueness and
some monotonicity of the minimizer $\psi_{\lambda, \th, \mu}$ to the
variational problem ($P_{\ld,\th,\mu}$), and we omit the proof here.
\begin{prop}\label{bb9}
For any $\lambda\in(0,+\infty)$ and $\th\in[0,\pi]$, the
minimizer $\psi_{\lambda, \th, \mu}$ to the truncated variational
problem ($P_{\ld,\th,\mu}$) is unique. Furthermore, the solution $\psi_{\ld, \th, \mu}$ is monotonic with respect to $x$, namely
\be\label{b19}\psi_{\lambda, \th, \mu}(x,r)\leq\psi_{\lambda, \th,
\mu}(\tilde{x},r)\ \ \ \ \text{for any}~~x\geq\tilde{x}.\ee
\end{prop}

\subsection{Some properties of the free boundaries}
\subsubsection{Preliminaries}

Before we investigate the properties of the free boundaries, we give
some important auxiliary lemmas, and we refer the proofs in
\cite{AC1, ACF1, FA1}. So we only state the result and omit the proof as follows.
\begin{lemma}\label{cc1} There exists a universal constant $c>0$, such that for $X_0=(x_0,r_0)\in\O_\mu\cap\{\psi_{\ld, \th,\mu}<0\}$ and $B_r(X_0)\subset\O_\mu\cap\{\psi_{\ld, \th,\mu}<0\}$ with
$$\f{1}{r}\fint_{\p B_r(X_0)}\left(\psi_{\ld,\th,\mu}-m_2\right)dS\geq\sqrt{\lambda} cr_0,$$ then we have $\psi_{\ld, \th,\mu}>m_2$ in $B_r(X_0)$; Similarly, $X_0=(x_0,r_0)\in\O_\mu\cap\{\psi_{\ld, \th,\mu}>0\}$ and $B_r(X_0)\subset\O_\mu\cap\{\psi_{\ld, \th,\mu}>0\}$, if $$\f{1}{r}\fint_{\p B_r(X_0)}\left(m_1-\psi_{\ld, \th,\mu}\right) dS\geq\sqrt{\lambda+\Ld} cr_0,$$ then we have $\psi_{\ld, \th,\mu}<m_1$ in $B_r(X_0)$. Here and after, $B_r(X)$ denotes some ball with radius $r>0$ and center
$X\in\O_{\mu}$.\end{lemma}

Next, we will establish a non-degeneracy lemma to $\psi_{\ld, \th,
\mu}-m_2$ and $m_1-\psi_{\ld, \th, \mu}$ as follows.
\begin{lemma}\label{cc2} (Non-degeneracy lemma) For any $0<\kappa_1<1$, there exists a positive constant $c$ (depending on $\kappa_1$),  if $B_r(X_0)\subset\O_\mu\cap\{\psi_{\ld, \th,\mu}<0\}$ ($X_0=(x_0,r_0)$)
and
$$\f{1}{r}\fint_{\p B_r(X_0)}\left(\psi_{\ld,\th,\mu}-m_2\right) dS\leq\sqrt{\lambda} cr_0,~~and~~\psi_{\ld, \th,\mu}<0~~in~~B_r(X_0),$$ then $\psi_{\ld, \th,\mu}=m_2$ in $B_{\kappa_1 r}(X_0)$;
Similarly, for any $0<\kappa_2<1$, there exists a positive constant
$c$ (depending on $\kappa_2$) and
$$\f{1}{r}\fint_{\p B_r(X_0)}\left(m_1-\psi_{\ld, \th, \mu}\right)
dS\leq\sqrt{\lambda+\Ld} cr_0,~~and~~\psi_{\ld, \th, \mu}>0~~in~~B_r(X_0),$$ then
$\psi_{\ld, \th,\mu}=m_1$ in $B_{\kappa_2 r}(X_0)$.\end{lemma}A direct
application of Lemma \ref{cc2} gives the following lemma.
\begin{lemma}\label{cc3}Suppose that $X_0=(x_0,r_0)\in \overline{\{\psi_{\ld, \th,\mu}>m_2\}\cap(\O_\mu\backslash D_{\mu})}$ and $\psi_{\ld, \th,\mu}<0$ in $B_r(X_0)$ for some $r>0$,
 then \be\label{b13}\f{1}{r}\fint_{\p B_r(X_0)}\left(\psi_{\ld,\th,\mu}-m_2\right) dS\geq\sqrt{\lambda} cr_0.\ee In
 particular,
 \be\label{b14}\sup_{\p B_r(X_0)}\left(\psi_{\ld,\th,\mu}-m_2\right)\geq \sqrt{\lambda}c r_0
 r. \ee\end{lemma}

We shall establish a non-oscillation lemma, which implies that the
free boundary $\Gamma_{i, \mu}$ for $i=1, 2$ cannot oscillate near
the solid boundaries. Without loss of generality, consider the right
free boundary $\Gamma_{2, \mu}$, and introduce a domain $G\subset
\O_\mu\backslash D_{\mu}$ bounded by
$$x=x_1,\quad x=x_1+h \ (h>0),$$and$$ \gamma_1:X=X^1(t)=(x^1(t),r^1(t)),\quad\gamma_2:X=X^2(t)=(x^2(t),r^2(t)),$$
where $0\leq t\leq T$ with $$x_1<x^i(t)<x_1+h \ \ \text{for}\ \
0<t<T,
$$ and $$x^i(0)=x_1,\ \  x^i(T)=x_1+h,\ \ r_1\leq r^i(t)\leq r_1+\delta, \ \
i=1,2.$$

Furthermore, the arc $\gamma_2$ lies above the arc $\gamma_1$, this
implies that $r^1(0)<r^2(0)$, $\gamma_1$ and $\gamma_2$ do not
intersect, $\gamma_2$ is contained in $\Gamma_{2,\mu}$, either
$$\text{Case}\ 1.\ \gamma_1 \ \ \text{is contained in}\ \Gamma_{2,\mu},\quad \text{(see
Figure \ref{fi7})}$$ or
$$\text{Case}\ 2.\ \gamma_1 \ \ \text{lies on}\ \{r=R,x>1\}, \ \text{and then} \ \ r_1=R,\  x_1\geq 1.\quad \text{(see Figure \ref{fi8})}$$
\begin{figure}
\begin{minipage}[t]{0.5\linewidth}
\centering
\includegraphics[width=60mm,height=35mm]{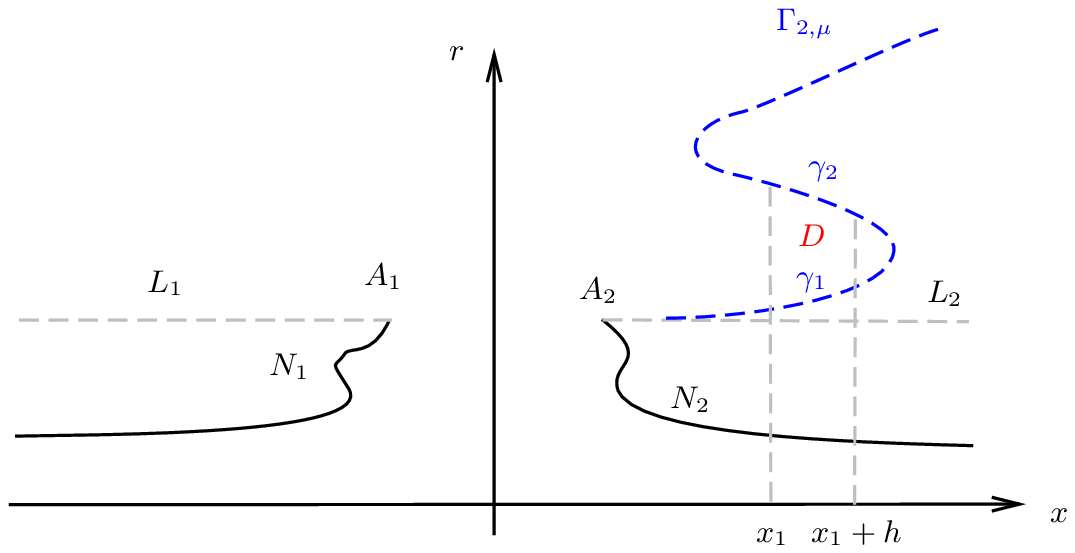}
\caption{Case 1}\label{fi7}
\end{minipage}%
\begin{minipage}[t]{0.5\linewidth}
\centering
\includegraphics[width=60mm,height=35mm]{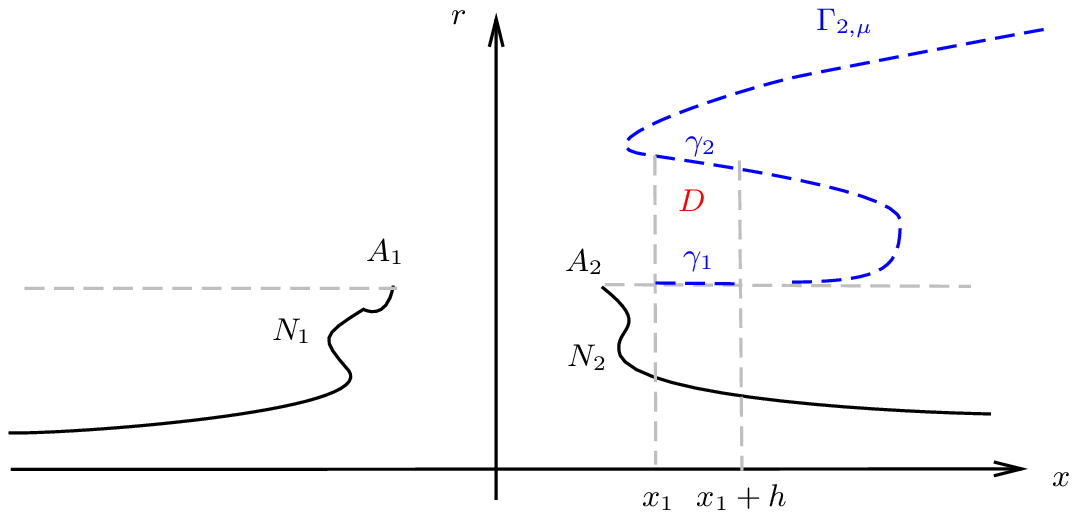}
\caption{Case 2}\label{fi8}
\end{minipage}
\end{figure}

Let the domain $G\subset\{\psi_{\ld, \th, \mu}>m_2\}$ be a
neighborhood of $\gamma_1$ and $\gamma_2$, and $\psi_{\ld, \th, \mu}<0$
in $G$ and for some $c^*>0$, we have
$$\text{dist}(G,\overline{A_1A_2})>c^*.$$

\begin{lemma}\label{cc5}(Non-oscillation lemma) Under the foregoing assumptions, there exists a  positive constant $C$ depending only on
$\lambda$, $m_2$ and $c^*$ such that \be\label{b15}h\leq
C\delta.\ee
\end{lemma}

The proof is similar to Lemma 4.1 in \cite{ACF1} and Lemma 5.6
\cite{ACF3}, we omit it here.

Finally, we give the uniform bound of the gradient to the minimizer,
which is independent of $m_1$ and $m_2$. Please see Lemma 8.1 in \cite{ACF3} and Lemma 5.2 in \cite{ACF1} for the proof.
\begin{lemma}\label{cc7}
Let $X_0=(x_0,r_0)$ be a free boundary point in
$\O_{\mu}\setminus\overline{D}_{\mu}$ and $G$ be a bounded domain
with $X_0\in G$, $\overline{G}\subset
{\O_\mu\setminus\overline{D}_{\mu}}$. There exists a constant
$C>0$ depending only on $\lambda$, $G$ and $\Ld$, such that
\be\label{b17}\f{\left|\g \psi_{\lambda, \th,\mu}\right|}{r}\leq C~~
\text{in}~~G.\ee
\end{lemma}

\subsubsection{Some properties of the free boundaries}\label{ss}

It follows from the monotonicity of $\psi_{\ld, \th, \mu}$ with
respect to $x$ that the free boundaries are $r$-graph, namely, the
free boundaries $\Gamma_{i,\mu}$ ($i=1,2$) intersect $r=r_0$ either one single
point or a segment for any $r_0\in (R, +\infty)$. Thus, there
exist four mappings $g_{1, \lambda, \th,\mu}(r)$ with $r>R$,
$g_{2, \lambda, \th, \mu}(r)$ with $r>R$, $g_{\lambda, \th,\mu}(r)$ with $r>0$ and $\t g_{\lambda, \th,\mu}(r)$ with $r>0$ such
that\be\label{b31}\left\{0<\psi_{\lambda, \th,
\mu}<m_1\right\}\cap\O_\mu=\left\{\tilde{g}_{1, \lambda, \th, \mu}(r)<x<g_{\lambda, \th,\mu}(r)\right\}\cap\O_\mu,\ee
and
\be\label{b031}\left\{m_2<\psi_{\lambda, \th,
\mu}<0\right\}\cap\O_\mu=\left\{\t g_{\lambda, \th,\mu}(r)<x<\tilde{g}_{2, \lambda, \th, \mu}(r)\right\}\cap\O_\mu,\ee
where
$$\begin{aligned}\tilde{g}_{1, \lambda, \th,
\mu}(r)=\left\{\ba{ll}f_1(r) \ \ \ &\text{for}~ 0<r\leq R,
\\ g_{1, \ld,\th,\mu}(r)\ \ \
\ &\text{for }~R< r<+\infty,\ea\right.\end{aligned}$$
and $$\begin{aligned}\tilde{g}_{2, \lambda, \th,
\mu}(r)=\left\{\ba{ll}f_2(r) \ \ \ &\text{for}~ 0<r\leq R,
\\ g_{2, \ld,\th,\mu}(r)\ \ \
\ &\text{for }~R< r<+\infty.\ea\right.\end{aligned}$$

Indeed, along similar arguments as in \cite{ACF1}, we obtain that $g_{i, \ld, \th, \mu}(r)$ is indeed a
general continuous function in $[R, +\infty)$, and $g_{i,\ld,\th,\mu}(R)$ is defined as $\lim_{r\rightarrow R^+}g_{i,\ld,\th,\mu}(r)$ for $i=1,2$. Furthermore, due to Lemma 3.3 in \cite{ACF6} and Proposition 4.1 in \cite{WX}, the interface $g_{\ld,\th,\mu}(r)\equiv\t g_{\ld,\th,\mu}(y)$ is indeed a continuous function
in $[0, +\infty)$, and we omit the proof here.
\begin{lemma}\label{dd02} The free boundary $\Gamma_{i,
\mu}:x=g_{i, \lambda, \th,\mu}(r)$ is a generalized continuous
function in $R\leq r<+\infty$ with values in $[-\infty,+\infty]$ ($i=1,2$), respectively. Furthermore, the interface $\Gamma_\mu$: $x=g_{\ld,\th,\mu}(r)$ is bounded continuous functions in $0<r<+\infty$, $g_{\ld,\th,\mu}(0+0)\triangleq\lim_{r\rightarrow 0^+}g_{\ld,\th,\mu}(r)$ exists and is finite.
\end{lemma}

In order to study the limit behavior of the solution as
$r\rightarrow+\infty$, we first establish the decay estimate of the
minimizer in far field as follows. This is one of the crucial parts
in this paper.
\begin{lemma}\label{lem1}For any $\theta\in(0,\pi)$ and $r_0>2R$, there exists a constant $C$ (independent of $r_0$) such that \be\label{aa1}\int_{\O_\mu\cap\{r>r_0\}}r\left|\f{\g\psi_{\ld,\th,\mu}}{r}-\left(\sqrt{\Lambda+\ld}\chi_{\{0<\psi_{\ld,\mu,\th}<m_1\}}+\sqrt{\ld}\chi_{\{m_2<\psi_{\ld,\th,\mu}\leq0\}}\right)e\right|^2dX\leq \f{C}{r_0^3}.\ee \end{lemma}
\begin{proof}
Denote $\psi(x,r)=\psi_{\ld,\th,\mu}(x,r)$, $g(r)=g_{\ld,\th,\mu}(r)$ and $g_{i}(r)=g_{i,\ld,\th,\mu}(r)$ ($i=1,2$) for simplicity.

For any $r_0>2R$, define $$S(r_0)=\int_{\O_\mu\cap\{\f{r_0}{2}<r<r_0\}}r\left|\f{\g\psi_{\ld,\th,\mu}}{r}-\left(\sqrt{\Lambda+\ld}\chi_{\{m_2<\psi_{\ld,\mu,\th}<0\}}+\sqrt{\ld}\chi_{\{m_2<\psi_{\ld,\th,\mu}\leq0\}}\right)e\right|^2dX.$$ Taking advantage of the mean value theorem, there exists some $\t r\in\left(\f{r_0}{2},r_0\right)$ such that \be\label{aa0}\begin{aligned}S(r_0)=&\f{r_0}{2\t r}\left\{\int_{g_1(\t r)}^{g(\t r)}\left|\g\psi(x,\t r)-\sqrt{\Lambda+\ld} \t re\right|^2dx+\int_{g(\t r)}^{g_2(\t r)}\left|\g\psi(x,\t r)-\sqrt{\ld} \t re\right|^2dx\right\}\\ \geq&\f12\left\{\int_{g_1(\t r)}^{g(\t r)}\left|\g\psi(x,\t r)-\sqrt{\Lambda+\ld} \t re\right|^2dx+\int_{g(\t r)}^{g_2(\t r)}\left|\g\psi(x,\t r)-\sqrt{\ld} \t re\right|^2dx\right\}.\end{aligned}\ee

We choose a function $w(x,r)$ as follows
\be\label{aa2}w(x,r)=\left\{\ba{ll}\psi(x,r),\ &\text{in}\ \
\O_\mu\cap\{r\leq \t r\},\\
\eta(r)\bar\psi(x,r)+(1-\eta(r))\phi(x,r),\ &\text{in}\ \ \O_\mu\cap\{r\geq \t r\},\ea\right.\ee where $\eta(r)=\max\left\{0,\f{\bar r-r}{\bar r-\t r}\right\}$ with $\bar r=\t r+\f {1}{\t r}$, $$\bar\psi(x,r)=\psi(x-(r-\t r)\cot\th,\t r),$$ and $$\begin{aligned}\phi(x,r)=&\min\left\{\max\left\{\sqrt{\Lambda+\ld} r\left((r-\t r)\cos\th-(x-g(\t r))\sin\th\right),0\right\},m_1\right\}\\&+\max\left\{\min\left\{\sqrt{\ld} r\left((r-\t r)\cos\th-(x-g(\t r))\sin\th\right),0\right\},m_2\right\}.\end{aligned}$$

It's easy to check that $w(x,r)\in K_\mu$, then $J_{\ld,\th,\mu}(\psi)\leq J_{\ld,\th,\mu}(w)$, which implies \be\label{aa3}\begin{aligned}&\int_{\O_\mu\cap\{r>\t r\}}r\left|\f{\g\psi_{\ld,\th,\mu}}{r}-\left(\sqrt{\Lambda+\ld}\chi_{\{m_2<\psi_{\ld,\mu,\th}<0\}}+\sqrt{\ld}\chi_{\{m_2<\psi_{\ld,\th,\mu}\leq0\}}\right)e\right|^2dX\\ \leq&\int_{\O_\mu\cap\{r>\t r\}}r\left|\f{\g w}{r}-\left(\sqrt{\Lambda+\ld}\chi_{\{m_2<w<0\}}+\sqrt{\ld}\chi_{\{m_2<w\leq0\}}\right)e\right|^2dX\\=&\int_{\O_\mu\cap\{\t r< r\leq \bar r\}}r\left|\f{\g w}{r}-\left(\sqrt{\Lambda+\ld}\chi_{\{m_2<w<0\}}+\sqrt{\ld}\chi_{\{m_2<w\leq0\}}\right)e\right|^2dX\\&+\int_{\O_\mu\cap\{r\geq \bar r\}}r\left|\f{\g w}{r}-\left(\sqrt{\Lambda+\ld}\chi_{\{m_2<w<0\}}+\sqrt{\ld}\chi_{\{m_2<w\leq0\}}\right)e\right|^2dX.\end{aligned}\ee

First, similar arguments as Lemma 3.10 in \cite{CXF} and Lemma 4.1 in \cite{CDW}. We obtain
\be\label{aa015}\int_{\O_\mu\cap\{r>\t r\}}r\left|\f{\g\psi}{r}-\left(\sqrt{\Lambda+\ld} \chi_{\{0<\psi<m_1\}}+\sqrt{\ld} \chi_{\{m_2<\psi\leq0\}}\right)e\right|^2dX\leq\f{C}{r_0^3}+\f{S(r_0)}{16},\ee where $C$ is independent of $r_0$.

Next, $S(2r_0)$ can be calculated as follows,
\be\label{aa16}\begin{aligned}S(2r_0)=&\int_{\O_\mu\cap\{r_0< r<2r_0\}}r\left|\f{\g\psi}{r}-\left(\sqrt{\Lambda+\ld} \chi_{\{0<\psi<m_1\}}+\sqrt{\ld} \chi_{\{m_2<\psi\leq0\}}\right)e\right|^2dX\\ \leq&\int_{\O_\mu\cap\{r> \t r\}}r\left|\f{\g\psi}{r}-\left(\sqrt{\Lambda+\ld} \chi_{\{0<\psi<m_1\}}+\sqrt{\ld} \chi_{\{m_2<\psi\leq0\}}\right)e\right|^2dX\\ \leq&\f{C}{r_0^3}+\f{S(r_0)}{16},\end{aligned}\ee where we have used $\t r\in\left(\f{r_0}{2},r_0\right)$ and \eqref{aa015}.

Using mathematical induction for any $n\in \mathbb{N}$ and \eqref{aa16}, one has \be\label{aa17}S(2^{n+1}R)\leq\f{2C}{\left(2^{n}R\right)^3},\ \ n=0,1....\ee

Indeed, \eqref{aa17} holds for $n=0$ when choose $C$ large enough. If \eqref{aa17} holds for $n$, one has $$S(2^{n+2}R)=S(2\cdot2^{n+1}R)\leq \f{C}{(2^{n+1}R)^3}+\f{S(2^{n+1}R)}{16}\leq \f{C}{(2^{n+1}R)^3}+\f{1}{16}\f{2C}{(2^nR)^3}=\f{2C}{\left(2^{n+1}R\right)^3},$$ which implies \eqref{aa17} holds for $n+1$.

Therefore, for any $r_0>2R$, there exists a $n_0$ such that $2^{n_0}R\leq r_0\leq 2^{n_0+1}R$, this together with \eqref{aa17} yields to $$\begin{aligned}&\int_{\O_\mu\cap\{r>r_0\}}r\left|\f{\g\psi}{r}-\left(\sqrt{\Lambda+\ld} \chi_{\{0<\psi<m_1\}}+\sqrt{\ld} \chi_{\{m_2<\psi\leq0\}}\right)e\right|^2dX\\ \leq&\int_{\O_\mu\cap\{r>2^{n_0}R\}}r\left|\f{\g\psi}{r}-\left(\sqrt{\Lambda+\ld} \chi_{\{0<\psi<m_1\}}+\sqrt{\ld} \chi_{\{m_2<\psi\leq0\}}\right)e\right|^2dX\\ \leq&\sum_{j=n_0}^{+\infty}S(2^{j+1}R)\leq \f{\t C}{R^3r_0^3},\end{aligned}$$ where have used the following fact $$\\ \sum_{j=n_0}^{+\infty}S(2^{j+1}R)\leq 2C\sum_{j=n_0}^{+\infty}\f{1}{\left(2^{j}R\right)^3}\leq\f{32C}{\left(2^{n_0+1}R\right)^3}\leq \f{\t C}{r_0^3}.$$

This completes the proof of Lemma \ref{lem1}.

\end{proof}

Firstly, we will show that some convergence of the minimizer in far field and the free boundaries approach to
the asymptotic direction $\th\in(0,\pi)$ as $r\rightarrow+\infty$ in the far field.
\begin{lemma}\label{lem2}Let $\th\in(0,\pi)$, $\psi_n(\t x,\t r)=\psi_{\ld,\th,\mu}\left(x_n+\f{\t x}{r_n}, r_n+\f{\t r}{r_n}\right)$ with $X_n=(x_n, r_n)\in\Gamma_{1, \mu}$ and $r_n\rightarrow+\infty$, $\t X=(\t x, \t r)\in\mathbb{R}^2$, then for a subsequence\be\label{le5}\begin{aligned}\psi_n(\t x,\t r)\rightarrow\Theta(\t x,\t r)\triangleq\left\{\ba{ll} m_1, \ \ \ &\text{if}\quad
 \t r\cos\th-\t x\sin\th\geq0,
\\ m_1+\sqrt{\Lambda+\ld}(\t r\cos\th-\t x\sin\th),\
\ &\text{if }\quad-\f{m_1}{\sqrt{\Lambda+\ld}}\leq \t r\cos\th-\t x\sin\th\leq 0,\\ \f{m_1\sqrt{\ld}}{\sqrt{\Lambda+\ld}}+\sqrt{\ld}(\t r\cos\th-\t x\sin\th),\
\ &\text{if }\quad \f{m_2}{\sqrt{\ld}}-\f{m_1}{\sqrt{\Lambda+\ld}}\leq \t r\cos\th-\t x\sin\th\leq -\f{m_1}{\sqrt{\Lambda+\ld}},\\
m_2,\ &\text{if}\quad\t r\cos\th-\t x\sin\th\leq
\f{m_2}{\sqrt{\ld}}-\f{m_1}{\sqrt{\Lambda+\ld}},\ea\right.\end{aligned}\ee uniformly in any
compact subset of $\mathbb{R}^2$. Furthermore,
$$g'_{1, \ld, \th, \mu}(r)\rightarrow\cot\th\ \ \text{as}\ \ r\rightarrow+\infty.$$The similar conclusion holds for $X_n\in\Gamma_{2, \mu}$.\end{lemma}
\begin{proof} Set $\psi=\psi_{\ld, \th,\mu}$, $g(r)=g_{\ld,\th,\mu}(r)$ and $g_i(r)=g_{i,\ld,\th,\mu}(r)$ ($i=1,2$) for simplicity. Define $x=x_n+\f{\t x}{r_n}$ and $r=r_n+\f{\t r}{r_n}$.

For any $R_0>0$, one has
\be\label{aa018}\begin{aligned}&\int_{\{|r-r_n|<R_0\}}r\left|\f{\g\psi}{r}-\left(\sqrt{\Lambda+\ld
}\chi_{\{0<\psi<m_1\}}+\sqrt{\ld
}\chi_{\{m_2<\psi\leq0\}}\right)e\right|^2dX\\=&\int_{\{|\t r|<R_0r_n\}}\left(\f{1}{r_n}+\f{\t r}{r_n^3}\right)\left|\f{\t\g\psi_n}{1+\f{\t r}{r_n^2}}-\left(\sqrt{\Lambda+\ld
}\chi_{\{0<\psi_n<m_1\}}+\sqrt{\ld
}\chi_{\{m_2<\psi_n\leq0\}}\right)e\right|^2d\t X,\end{aligned}\ee where $\t\g=(\p_{\t x}, \p_{\t r})$.

For large $r_n>R_0+2R$, in view of \eqref{aa1}, we obtain
$$\int_{\{|r-r_n|<R_0\}}r\left|\f{\g\psi}{r}-\left(\sqrt{\Lambda+\ld
}\chi_{\{0<\psi<m_1\}}+\sqrt{\ld
}\chi_{\{m_2<\psi\leq0\}}\right)e\right|^2dX \leq\f{C}{\left(r_n-R_0\right)^3},$$
which together with \eqref{aa018} implies that \be\label{aa18}\int_{\O_{\mu}\cap\{|\t r|<R_0r_n\}}\left(1+\f{\t r}{r_n^2}\right)\left|\f{\t\g\psi_n}{1+\f{\t r}{r_n^2}}-\left(\sqrt{\Lambda+\ld
}\chi_{\{0<\psi_n<m_1\}}+\sqrt{\ld
}\chi_{\{m_2<\psi_n\leq0\}}\right)e\right|^2d\t X\leq\f{Cr_n}{\left(r_n-R_0\right)^3}.\ee
 Recalling Proposition \ref{cc8}
and $\Theta(\t x,\t r)\in H_{loc}^1(\mathbb{R}^2)$, then there exist
a subsequence $\psi_{n_k}$ and two functions
$\gamma_1,\gamma_2\in[0,1]$ such that
$$\psi_{n_k}(\t x,\t r)\rightarrow\Theta(\t x, \t r)\ \
\text{weakly in $H_{loc}^1(\mathbb{R}^2)$},$$
$$\psi_{n_k}(\t x,\t r)\rightarrow\Theta(\t x, \t r)\ \ \text{a.e. in $\mathbb{R}^2$},$$
$$\chi_{\{0<\psi_{n_k}<m_1\}}\rightarrow\gamma_1\ \ \text{weak-star in $L_{loc}^\infty(\mathbb{R}^2)$, $\gamma_1=1$ a.e. on $\{0<\Theta(\t x,\t r)<m_1\}$},$$ and $$\chi_{\{m_2<\psi_{n_k}\leq0\}}\rightarrow\gamma_2\ \ \text{weak-star in $L_{loc}^\infty(\mathbb{R}^2)$, $\gamma_2=1$ a.e. on $\{m_2<\Theta(\t x,\t r)\leq0\}$},$$
as $k\rightarrow+\infty$. This together with
\eqref{aa18} gives that
\be\label{le2}\t\g \Theta=\sqrt{\Lambda+\ld}
e\chi_{\{0<\psi_0<m_1\}}+\sqrt{\ld}
e\chi_{\{m_2<\psi_0\leq0\}}\ \ \text{a.e.},\ee in any compact
subset $\Omega'$ of $\mathbb{R}^2$.

Lemma \ref{cc7} implies that for sufficiently large $n$
$$|\t\g\psi_n(\t x,\t r)|=\left|\f{1}{r_n}\g\psi\left(x_n+\f{\t
x}{\t r_n}, r_n+\f{\t r}{r_n}\right)\right|\leq c_0,$$ where the
constant $c_0$ is independent of $R_0$. Hence, we conclude that there
exists a subsequence $\psi_{n_k}\rightarrow\Theta(\t x,\t r)$
uniformly in any compact subset of $\mathbb{R}^2$ and
$$m_1-\psi_n(\t x, \t r)=\psi(x_n,r_n)-\psi\left(x_n+\f{\t x}{\t r_n}, r_n+\f{\t r}{r_n}\right)\leq|\t\g\psi_n||\t X|,\ \ \text{for}\ \ |\t X|<\f{m_1}{c_0},$$which implies $\psi_n(\t x, \t r)>0$.

The non-degeneracy lemma \ref{cc2} implies that $$\f{1}{r}\fint_{\p
B_r(0)}(m_1-\psi_n(\t x,\t r))d\t S=\f{1}{r_n}\f{1}{\f{r}{r_n}}\fint_{\p
B_{\f{r}{r_n}}(X_n)}(m_1-\psi(x,r))d S\geq c\sqrt{\Lambda+\ld}, \ \ \text{for}\ \
r<\f{m_1}{c_0},$$ taking $n\rightarrow+\infty$, which implies that $\Theta\not\equiv m_1$ in
$B_r(0)$ and $\Theta(0)=m_1$.

Define \be\label{le4}
 t=\t x\cos\th+\t r\sin\th\ \ \text{and}\ \ s=\t r\cos\th-\t x\sin\th,\ee
and $w(t, s)=\Theta(\t x,\t r)$, then \eqref{le2} implies that
$$\f{\p w}{\p t}=0, \ \ \f{\p w}{\p s}=\sqrt{\Lambda+\ld}
e\chi_{\{0<w<m_1\}}+\sqrt{\ld} e\chi_{\{m_2<w\leq0\}}\ \ \text{a.e.
in $\O'_\mu$}\ \ \text{and}\ \ w(0)=m_1.$$

A direction computation gives that
$$\begin{aligned}w(t,s)=\left\{\ba{ll} m_1 \ \ \ &\text{if}~~ s\geq0,
\\ m_1+\sqrt{\Lambda+\ld} s,\ \ \
\ &\text{if }~-\f{m_1}{\sqrt{\Lambda+\ld}}\leq s\leq 0,\\ \f{m_1\sqrt{\ld}}{\sqrt{\Lambda+\ld}}+\sqrt{\ld}s,\
\ &\text{if }\quad \f{m_2}{\sqrt{\ld}}-\f{m_1}{\sqrt{\Lambda+\ld}}\leq s\leq -\f{m_1}{\sqrt{\Lambda+\ld}},\\
m_2,\ \ \ &\text{if}~~s\leq
\f{m_2}{\sqrt{\ld}}-\f{m_1}{\sqrt{\Lambda+\ld}},\ea\right.\end{aligned}$$
which yields \eqref{le5}.

Next, let $X_0=(t,s)$ with $s>0$ or $s<\f{m_2}{\sqrt{\ld}}-\f{m_1}{\sqrt{\Lambda+\ld}}$, for small $r>0$, then
$$\lim_{n\rightarrow+\infty}\f{1}{r}\fint_{\p B_r(X_0)}(m_1-\psi_n)dS=0,\ \ \text{or}\ \ \lim_{n\rightarrow+\infty}\f{1}{r}\fint_{\p B_r(X_0)}(\psi_n-m_2)dS=0,$$
respectively. Then, the non-degeneracy lemma implies that $X_0$ is
not a free boundary point for sufficiently large $n$.

Similarly, for the case $\f{m_2}{\sqrt{\ld}}-\f{m_1}{\sqrt{\Lambda+\ld}}<s<0$, one gets
$$\lim_{n\rightarrow+\infty}\f{1}{r}\fint_{\p B_r(X_0)}(m_1-\psi_n)dS\rightarrow+\infty,\ \
\text{or}\ \ \lim_{n\rightarrow+\infty}\f{1}{r}\fint_{\p
B_r(X_0)}(\psi_n-m_2)dS\rightarrow+\infty,\ \ \text{as}\ \
r\rightarrow0,$$ and then $X_0$ is not a free boundary point for sufficiently large $n$.

Then, one has \be\label{le3}\p\{\psi_n>0\}\rightarrow
\left\{s=\f{m_2}{\sqrt{\ld}}-\f{m_1}{\sqrt{\Lambda+\ld}}\right\} \ \ \text{and}\ \
\p\{\psi_n<m_1\}\rightarrow \left\{s=0\right\},\ee locally in
Hausdorff distance (see Definition 3.1 in \cite{FA1}).

Noticing the flatness
conditions in Section 7 in \cite{AC1} for the free boundaries, there exists a $\xi_{1,n}\in\left(\min\left\{r_n, r_n+\f{\t r}{r_n}\right\},\max\left\{r_n, r_n+\f{\t r}{r_n}\right\}\right)$ such that
$$\t x=r_n\left(g_{1}\left(r_n+\f{\t r}{r_n}\right)-x_n\right)=r_n\left(g_{1}\left(r_n+\f{\t r}{r_n}\right)-g_{1}(r_n)\right)=g_1'(\xi_{1,n})\t r,$$ thus, we obtain
$$g_{1}'\left(r_n+\f{\t r}{r_n}\right)\rightarrow\cot\th,\ \ \text{as}\ \ n\rightarrow+\infty.$$
Therefore, we complete the proof of Lemma \ref{lem2}.
\end{proof}

Next, for the critical cases $\th=0$ or $\th=\pi$, we have the following facts.
\begin{prop}\label{co2}
Assume that there exist some free boundary points
$X_n=(x_n,r_n)\in\Gamma_{2,\mu}$, such that $r_n\rightarrow\xi>R$ and
$x_n\rightarrow +\infty$, where $\xi$ is a finite positive number, then $\th=0$. Moreover, let $\psi_n(\t X)=\psi_{\ld,\th,\mu}(x_n+\t x, r_n+\t r)$,
then $$
\psi_n(\t X)\rightarrow\min\left\{\max\left\{\sqrt{\ld}\left(\f{m_2}{\sqrt{\ld}}+\f{\t r^2}{2}+\xi\t r\right) , m_2\right\},
0\right\}+\max\left\{\min\left\{\sqrt{\Lambda+\ld}\left(\f{m_2}{\sqrt{\ld}}+\f{\t r^2}{2}+\xi\t r\right), m_1\right\},
0\right\},$$ uniformly in any compact subset of $\{(\t x,\t r)\mid
\t r>R-\xi\}$. If
$r_n\rightarrow \xi>\sqrt{R^2+\f{2m_1}{\sqrt{\Lambda+\ld}}-\f{2m_2}{\sqrt{\ld}}}$ and $x_n\rightarrow -\infty$, $\xi$ is a
finite positive number, then $\th=\pi$ and $$
\psi_n(\t X)\rightarrow\min\left\{\max\left\{\sqrt{\ld}\left(\f{m_2}{\sqrt{\ld}}-\f{\t r^2}{2}-\xi\t r\right) , m_2\right\},
0\right\}+\max\left\{\min\left\{\sqrt{\Lambda+\ld}\left(\f{m_2}{\sqrt{\ld}}-\f{\t r^2}{2}-\xi\t r\right), m_1\right\},
0\right\},$$ uniformly in any compact subset of $\{(\t x,\t r)\mid
\t r>R-\xi\}$. The similar assertion holds for $X_n=(x_n,r_n)\in\Gamma_{1,\mu}$.
\end{prop}

\begin{proof}If
$X_n=(x_n,r_n)\in\Gamma_{2, \mu}$ with $r_n\rightarrow \xi$ ($\xi$
is a finite positive number) and $x_n\rightarrow +\infty$. Set $x=x_n+\t x$ and $r=r_n+\t r$. For any large $R_0>0$, the boundedness of $J_{\ld,\th,\mu}(\psi_{\ld,\th,\mu})$ gives that $$\begin{aligned}&\int_{\O_\mu\cap\{|x-x_n|<R_0\}\cap\{R-\xi<r-r_n <R_0\}}r\left|\f{\g\psi}{r}-\left(\sqrt{\Lambda+\ld
}\chi_{\{0<\psi<m_1\}}+\sqrt{\ld
}\chi_{\{m_2<\psi\leq0\}}\right)e\right|^2dX\\=&\int_{\t\O_\mu\cap\{|\t x|<R_0\}\cap\{R-\xi<\t r<R_0\}}\left(r_n+\t r\right)\left|\f{\t\g\psi_n}{r_n+\t r}-\left(\sqrt{\Lambda+\ld
}\chi_{\{0<\psi_n<m_1\}}+\sqrt{\ld
}\chi_{\{m_2<\psi_n\leq0\}}\right)e\right|^2d\t X\rightarrow0,\end{aligned}$$ as $n\rightarrow+\infty$.

Along the
similar arguments in Lemma \ref{lem2}, then there exists a subsequence
$\psi_{n_k}$ such that
$$\psi_{n_k}\rightarrow\psi_0\ \
\text{weakly in $H_{loc}(\O'_\mu)$},$$ and
$$\psi_{n_k}\rightarrow\psi_0\ \ \text{a.e. in $\O'_\mu$},$$as $k\rightarrow+\infty$, and $$\g \psi_0=(\t r+\xi)
\left(\sqrt{\Lambda+\ld
}\chi_{\{0<\psi_0<m_1\}}+\sqrt{\ld
}\chi_{\{m_2<\psi_0\leq0\}}\right)(-\sin\th, \cos\th)\ \ \text{a.e.},$$ in
any compact subset $\Omega'$ of $\{(\t x,\t r)\mid \t r>R-\xi\}$.
Furthermore, it's easy to see that $\psi_0(0,0)=m_2$,
$\psi_0\not\equiv m_2$ in any neighborhood of $(0,0)$ and
\be\label{c0}\psi(x_n+\t x, R-\xi+r_n)\rightarrow\psi_0(\t
x,R-\xi)=m_2\ \ \text{if}\ \ x_n\rightarrow+\infty.\ee

Next, we claim $\th=0$. Suppose not, if $\theta=\pi$, indeed, similar arguments as Lemma
\ref{lem2}, we obtain
$$\psi_0(\t x,\t r)=\min\left\{\max\left\{\sqrt{\ld}\left(\f{m_2}{\sqrt{\ld}}-\f{\t r^2}{2}-\xi\t r\right) , m_2\right\},
0\right\}+\max\left\{\min\left\{\sqrt{\Lambda+\ld}\left(\f{m_2}{\sqrt{\ld}}-\f{\t r^2}{2}-\xi\t r\right), m_1\right\},
0\right\},$$ which contradicts
with \eqref{c0}.

If $\th\in(0,\pi)$, since $\psi_0$ is smooth in any compact subset of $G\subset\{m_2<\psi_0<0\}\cap\{\t r\geq R-\xi\}$, one has $$\f{\p^2\psi_0}{\p\t x\p\t r}=-\sqrt{\ld}\sin\th,\ \ \f{\p^2\psi_0}{\p\t r\p\t x}=0,$$ which derives a contradiction with $\th\in(0,\pi)$.

Therefore, we obtain $\th=0$. Along the similar
arguments in Lemma \ref{lem2}, one has
$$
\psi_n(\t X)\rightarrow\min\left\{\max\left\{\sqrt{\ld}\left(\f{m_2}{\sqrt{\ld}}+\f{\t r^2}{2}+\xi\t r\right) , m_2\right\},
0\right\}+\max\left\{\min\left\{\sqrt{\Lambda+\ld}\left(\f{m_2}{\sqrt{\ld}}+\f{\t r^2}{2}+\xi\t r\right), m_1\right\},
0\right\},$$ uniformly in any compact subset of $\{(\t x,\t r)\mid
\t r>R-\xi\}$. The similar conclusion holds if
$X_n=(x_n,r_n)\in\Gamma_{1,\mu}$ with $r_n\rightarrow \xi>R$ and
$x_n\rightarrow -\infty$.

Similarly, we can obtain the conclusion for $\th=\pi$. Thus, we complete the proof of Proposition \ref{co2}.

\end{proof}

Now, we can obtain the convergence rate of distance of the two free boundaries and the minimizer as follows.
\begin{lemma}\label{lem3}For any $\th\in(0,\pi)$ and $\alpha\in(0,2)$, the free boundaries $x=g_{1,\ld,\th,\mu}(r)$, $x=g_{2,\ld,\th,\mu}(r)$, the interface $x=g_{\ld,\th,\mu}(r)$ and the minimizer $\psi_{\ld,\th,\mu}$ satisfy \be\label{aa20}r(g_{\ld,\th,\mu}(r)-g_{1,\ld,\th,\mu}(r))\rightarrow\f{m_1}{\sqrt{\Lambda+\ld}\sin\th},\ee
 \be\label{aa020}r(g_{\ld,\th,\mu}(r)-g_{2,\ld,\th,\mu}(r))\rightarrow\f{m_2}{\sqrt{\ld}\sin\th},\ee
 and \be\label{aa21}r^\alpha\left(\f{\g\psi_{\ld,\th,\mu}}{r}-\left(\sqrt{\Lambda+\ld
}\chi_{\{0<\psi_{\ld,\th,\mu}<m_1\}}+\sqrt{\ld
}\chi_{\{m_2<\psi_{\ld,\th,\mu}\leq0\}}\right)e\right)\rightarrow0,\ee
as $r\rightarrow+\infty$.
\end{lemma}
\begin{proof}
Define $\psi_n(\t x,\t r)=\psi_{\ld,\th,\mu}\left(x_n+\f{\t x}{r_n},
r_n+\f{\t r}{r_n}\right)$ with $(x_n,
r_n)\in\Gamma_{1,\ld,\th,\mu}$, $r_n\rightarrow+\infty$. Set
$x=x_n+\f{\t x}{r_n}$ and $r=r_n+\f{\t r}{r_n}$. The free boundaries
and interface of $\psi_n(\t x,\t r)$ are given by \be\label{aa020}\left\{(\t x,\t
r)\mid x_n+\f{\t x}{r_n}=g_{i,\ld,\th,\mu}\left(r_n+\f{\t r}{r_n}\right),
i=1,2\right\}\ \ \text{and}\ \ \left\{(\t x,\t
r)\mid x_n+\f{\t x}{r_n}=g_{\ld,\th,\mu}\left(r_n+\f{\t r}{r_n}\right)\right\}.\ee

Thanks to Lemma \ref{lem2}, we have $$\text{$\p\{0<\psi_n<m_1\}$
converges to $\p\{0<\Theta<m_1\}$ locally in Hausdorff
distance.}$$
This together with \eqref{aa020} and \eqref{le5}, taking
$\t r=0$, yields that $$\t
x\sin\th=r_n\left(g_{\ld,\th,\mu}(r_n)-x_n\right)\sin\th=r_n\left(g_{\ld,\th,\mu}(r_n)-g_{1,\ld,\th,\mu}(r_n)\right)\sin\th\rightarrow\f{m_1}{\sqrt{\Lambda+\ld}}.$$
Furthermore, set $\psi_n(\t x,\t r)=\psi_{\ld,\th,\mu}\left(x_n+\t x,
r_n+\t r\right)$ with $(x_n,
r_n)\in\Gamma_{1,\ld,\th,\mu}$, $r_n\rightarrow+\infty$, for any $\t R>0$, and large $r_n>\t R+2R$, similarly in
Lemma \ref{lem1}, one has
$$\begin{aligned}&\int_{\O_{\mu}\cap\{|r-r_n|<\t R\}}r^{2\alpha}\left|\f{\g\psi_{\ld,\th,\mu}}{r}-\left(\sqrt{\Lambda+\ld
}\chi_{\{0<\psi_{\ld,\th,\mu}<m_1\}}+\sqrt{\ld
}\chi_{\{m_2<\psi_{\ld,\th,\mu}\leq0\}}\right)e\right|^2dX\\=&\int_{\{|\t r|<\t R\}}\left(\t r+r_n\right)^{2\alpha}\left|\f{\t\g\psi_n}{\t r+r_n}-\left(\sqrt{\Lambda+\ld
}\chi_{\{0<\psi_n<m_1\}}+\sqrt{\ld
}\chi_{\{m_2<\psi_n\leq0\}}\right)e\right|^2d\t X\\ \leq&\f{C(r_n+\t R)^{2\alpha-1}}{\left(r_n-\t R\right)^3}.\end{aligned}$$ Hence for any $\alpha\in(0,2)$ and $(\t x,\t r)\in \mathbb{R}^2$, one has $$\left(\t r+r_n\right)^{\alpha}\left|\f{\t\g\psi_n}{\t r+r_n}-\left(\sqrt{\Lambda+\ld
}\chi_{\{0<\psi_n<m_1\}}+\sqrt{\ld
}\chi_{\{m_2<\psi_n\leq0\}}\right)e\right|\rightarrow0\ \ \text{as
$n\rightarrow+\infty$},$$ taking $\t r=0$, $x=x_n+\t x$ and
$r=r_n$ yields to the desired estimate \eqref{aa21}.

Therefore, we complete the proof of Lemma \ref{lem3}.
\end{proof}

Next, we will prove that one of free boundaries will vanish, provided that
the asymptotic direction of the outgoing jet is horizontal. We call
that $\Gamma_{1,\mu}$ vanishes in $\O_\mu\cap\{r>R\}$, means $\psi_{\ld, \th, \mu}<m_1$
in $\O_\mu\cap\{r>R\}$, and similarly, we call that the free boundary $\Gamma_{2,\mu}$ vanishes in $\O_\mu\cap\{r>R\}$ means that $\psi_{\ld, \th, \mu}>m_2$
in $\O_\mu\cap\{r>R\}$.
\begin{prop}\label{co1} (1). If $\th=\pi$, then the left free boundary $\Gamma_{1, \mu}$ vanishes
in $\O_\mu\cap\{r>R\}$;

(2). If $\th=0$, then the
right free boundary $\Gamma_{2, \mu}$ vanishes in
$\O_\mu\cap\{r>R\}$.
%(3). If $\th=0$ and $R_1^2-R^2\geq \f{2m_1}{\sqrt{\ld+\Ld}}-\f{2m_2}{\sqrt{\ld}}$, then the
%left free boundary $\Gamma_{1, \mu}$ vanishes in
%$\O_\mu\cap\{r>R_1\}$.
\end{prop}

\begin{proof}Denote $\psi=\psi_{\ld,\th,\mu}$ for simplicity.

For $\th=\pi$, then $e=(0,-1)$, for $(x,r)\in\O_\mu$, set  $$\psi_0(x,r)=\max\left\{\min\left\{m_1-\f{\sqrt{\Lambda+\ld}(r^2-R^2)}{2},m_1\right\},0\right\}+\min\left\{\max\left\{\f{\sqrt{\ld}(R^2+\f{2m_1}{\sqrt{\Lambda+\ld}}-r^2)}{2},m_2\right\},0\right\}.$$

Next, we claim that \be\label{aa23}\psi(x,r)\leq\psi_0(x,r)\ \ \text{in $\O_\mu$}.\ee

 Suppose that the assertion \eqref{aa23} is not true, recalling that $\min\{\psi,\psi_0\}\in K_\mu$ and the uniqueness of minimizer, we obtain $$J_{\ld,\th,\mu}(\psi)<J_{\ld,\th,\mu}(\min\{\psi,\psi_0\}).$$This implies that there exists some sufficiently large $R_0>\max\left\{\sqrt{\f{2m_1}{\sqrt{\Lambda+\ld}}-\f{2m_2}{\sqrt{\ld}}+R^2}, 1\right\}$, and
\be\label{aa25}\begin{aligned}0>&\int_{\O_{\mu,R_0}}r\left|\f{\g\psi}{r}-\left(\sqrt{\Lambda+\ld
}\chi_{\{0<\psi<m_1\}}+\sqrt{\ld
}\chi_{\{m_2<\psi\leq0\}}\right)e\right|^2dX\\&-\int_{\O_{\mu,R_0}}r\left|\f{\g\min\{\psi,\psi_0\}}{r}-\left(\sqrt{\Lambda+\ld
}\chi_{\{0<\min\{\psi,\psi_0\}<m_1\}}+\sqrt{\ld
}\chi_{\{m_2<\min\{\psi,\psi_0\}\leq0\}}\right)e\right|^2dX\\=&\int_{\O_{\mu,R_0}}\f{\g\max\{\psi-\psi_0,0\}\cdot\g(\psi+\psi_0)}{r}dX
\\&-2\sqrt{\Lambda+\ld}\int_{\O_{\mu,R_0}}\g\psi\cdot
e\chi_{\{0<\psi<m_1\}}-\g\min\left\{\psi, \psi_0\right\}\cdot
e\chi_{\{0<\min\left\{\psi,
\psi_0\right\}<m_1\}}dX\\&-2\sqrt{\ld}\int_{\O_{\mu,R_0}}\g\psi\cdot
e\chi_{\{m_2<\psi\leq0\}}-\g\min\left\{\psi, \psi_0\right\}\cdot
e\chi_{\{m_2<\min\left\{\psi,
\psi_0\right\}\leq0\}}dX\\&+\int_{\O_{\mu,R_0}}(\Lambda+\ld)r\left(\chi_{\{0<\psi<m_1\}}-\chi_{\{0<\min\{\psi,\psi_0\}<m_1\}}\right)+\ld r\left(\chi_{\{m_2<\psi\leq0\}}-\chi_{\{m_2<\min\{\psi,\psi_0\}\leq 0\}}\right)dX\\=&I_1+I_2+I_3+I_4,\end{aligned}\ee where $\O_{\mu,R_0}$ is bounded by $N_{i,\mu}$, $N_{0,\mu}$, $L_i$, $H_{i,\mu}$, $\left\{((-1)^i,r)\mid R\leq r\leq R_0\right\}$
and $\{(x,R_0)\mid-R_0\leq x\leq R_0\}$ for $i=1,2$.

The first term $I_1$ can be estimated as follows,
\be\label{aa26}\begin{aligned}I_1=&\int_{\O_{\mu,R_0}}\f{|\g\max\{\psi-\psi_0,0\}|^2}{r}dX+2\int_{\O_{\mu,R_0}}\f{\g\max\{\psi-\psi_0,0\}\cdot\g\psi_0}{r}dX\\=&\int_{\O_{\mu,R_0}}\f{|\g\max\{\psi-\psi_0,0\}|^2}{r}dX-2\sqrt{\Lambda+\ld}\int_{\bar\O_{\mu,R_0}\cap\{\psi_0=0\}}\max\{\psi-\psi_0,0\}dx\\&+2\sqrt{\ld}\int_{\bar\O_{\mu,R_0}\cap\{\psi_0=0\}}\max\{\psi-\psi_0,0\}dx-2\sqrt{\ld}\int_{\bar\O_{\mu,R_0}\cap\{\psi_0<0\}\cap\{r=R_0\}}\max\{\psi-\psi_0,0\}dx.\end{aligned}\ee
Furthermore, for the second term $I_2$, one has \be\label{aa27}\begin{aligned}I_2=&-2\sqrt{\Lambda+\ld}\left\{\int_{\O_{\mu,R_0}\cap\{0<\psi_0<m_1\}\cap\{0<\psi<m_1\}}\g\max\left\{\psi-\psi_0,0\right\}\cdot edX\right.\\&+\left.\int_{\O_{\mu,R_0}\cap\{\psi_0<0\}\cap\{0<\psi<m_1\}}\g(\psi-\psi_0)\cdot edX-\int_{\O_{\mu,R_0}\cap\{0<\psi_0<m_1\}\cap\{\psi=m_1\}}\g(m_1-\psi_0)\cdot edX\right\}\\=&2\sqrt{\Lambda+\ld}\int_{\bar\O_{\mu,R_0}\cap\{\psi_0=0\}}\max\{\psi-\psi_0,0\}dx.\end{aligned}\ee
Similarly, we obtain \be\label{aa027}I_3=-2\sqrt{\ld}\int_{\bar\O_{\mu,R_0}\cap\{\psi_0=0\}}\max\{\psi-\psi_0,0\}dx+2\sqrt{\ld}\int_{\bar\O_{\mu,R_0}\cap\{\psi_0<0\}\cap\{r=R_0\}}\max\{\psi-\psi_0,0\}dx.\ee
Finally, we have \be\label{aa28}\begin{aligned}I_4\geq&(\Lambda+\ld)\int_{\O_{\mu,R_0}}r\left(\chi_{\{0<\psi<m_1\}\cap\{\psi_0\leq0\}}-\chi_{\{0<\psi_0<m_1\}\cap\{\psi=m_1\}}\right)dX-\ld\int_{\O_{\mu,R_0}}r\chi_{\{m_2<\psi_0\leq0\}\cap\{\psi>0\}}dX\\ \geq&-(\Lambda+\ld)\int_{\O_{\mu,R_0}}r\chi_{\{0<\psi_0<m_1\}\cap\{\psi=m_1\}}dX-\ld\int_{\O_{\mu,R_0}\cap\{m_2<\psi_0\leq0\}}r\left(\chi_{\{\psi>0\}}-\chi_{\{0<\psi<m_1\}}\right)dX\\ \geq&-(\Lambda+\ld)\int_{\O_{\mu,R_0}}r\chi_{\{0<\psi_0<m_1\}\cap\{\psi=m_1\}}dX-\ld\int_{\O_{\mu,R_0}}r\chi_{\{m_2<\psi_0\leq0\}\cap\{\psi=m_1\}}dX.\end{aligned}\ee

Inserting \eqref{aa26}-\eqref{aa28} into \eqref{aa25} yields

$$\begin{aligned}0>&\int_{\O_{\mu,R_0}}\f{|\g\max\{\psi-\psi_0,0\}|^2}{r}dX-(\Lambda+\ld)\int_{\O_{\mu,R_0}}r\chi_{\{0<\psi_0<m_1\}\cap\{\psi=m_1\}}dX\\&-\ld\int_{\O_{\mu,R_0}}r\chi_{\{m_2<\psi_0\leq0\}\cap\{\psi=m_1\}}dX\\=&\int_{\O_{\mu,R_0}\cap\{0<\psi_0<m_1\}\cap\{\psi=m_1\}}\left(\f{|\g\psi_0|^2}{r}- (\Lambda+\ld)r\right)dX+\int_{\O_{\mu,R_0}\cap\{0<\psi_0<\psi<m_1\}}\f{|\g(\psi-\psi_0)|^2}{r}dX\\&+\int_{\O_{\mu,R_0}\cap\{m_2<\psi_0\leq0\}\cap\{\psi=m_1\}}\left(\f{|\g\psi_0|^2}{r}-\ld r\right)dX+\int_{\O_{\mu,R_0}\cap\{m_2<\psi_0\leq0\}\cap\{\psi_0<\psi<m_1\}}\f{|\g(\psi-\psi_0)|^2}{r}dX\\=&\int_{\O_{\mu,R_0}\cap\{m_2<\psi_0<\psi<m_1\}}\f{|\g(\psi-\psi_0)|^2}{r}dX,\end{aligned}$$
which derives a contradiction. Hence, \eqref{aa23} holds, it implies that $$\psi(x,r)<m_1\ \ \text{in $\O_{\mu}\cap\{r>R\}$},$$ and thus, this gives that the free boundaries $\Gamma_{1,\mu}$ vanishes.

For $\theta=0$, taking $$\psi_0=\min\left\{\max\left\{\f{\sqrt{\ld}(r^2-R^2)}{2}+m_2,m_2\right\},0\right\}+\max\left\{\min\left\{\f{\sqrt{\Lambda+\ld}(r^2-R^2+\f{2m_2}{\sqrt{\ld}})}{2},m_1\right\},0\right\}.$$ Similar arguments as before, yield that $$\psi(x,r)\geq\psi_0(x,r)\ \ \text{in $\O_{\mu}\cap\{r>R\}$},$$ which implies that the free boundary $\Gamma_{2,\mu}$ is empty.

%Similarly, we can show that the result (3) in Proposition \ref{co1}
%holds.

Therefore, we complete the proof of Proposition \ref{co1}.
\end{proof}
%\begin{remark}In this paper, we assume that $R_1\geq R_2$ without loss of generality. In fact, for $R_1<R_2$, we can obtain the similar conclusion as follows\\ (1) The right free boundary $\Gamma_{2,\mu}$ vanishes for $\th=0$;\\ (2) The right free boundary $\Gamma_{2,\mu}$ vanishes for $R_2^2-R_1^2\geq \f{2m_1}{\sqrt{\ld+\Ld}}-\f{2m_2}{\sqrt{\ld}}$ and $\th=\pi$;\\ (3) The left free boundary $\Gamma_{1,\mu}$ vanishes for $R_2^2-R_1^2\leq \f{2m_1}{\sqrt{\ld+\Ld}}-\f{2m_2}{\sqrt{\ld}}$ and $\th=\pi$.\end{remark}

\begin{remark}\label{e00}Furthermore, we define
$g_{1, \lambda, \th,\mu}(R)=-\infty$ for $\th=\pi$,
$g_{2, \lambda, \th,\mu}(R)=+\infty$ for $\th=0$, respectively.
\end{remark}

Proposition \ref{co1} implies that the one of free boundaries
vanishes for horizontal asymptotic direction, and on another side,
we will show that the both of two free boundaries are non-empty, for
non-horizontal asymptotic direction.
\begin{lemma}\label{lee}If $\th\in(0,\pi)$, then $\Gamma_{i,\mu}$ is non-empty and a connected curve, $x=g_{i,\ld,\th,\mu}(r)$ is continuous in $(R, +\infty)$. And $\lim_{r\rightarrow R^+}g_{i,\ld,\th,\mu}(r)$ exists and denoted as $g_{i,\ld,\th,\mu}(R+0)$ for $i=1, 2$.\end{lemma}
\begin{proof}
{\bf Step 1.} We will show that $\Gamma_{i,\mu}$ is non-empty for $i=1,2$.

Firstly, we claim that there exists a constant $R_0>0$, such that
$$B_{R_0}(X_0)\subset\O_\mu\cap\{r>R_0\}$$  contains a free boundary point
$X_0=(x_0,R+R_0)\in\O_\mu$ for any  $0<\psi_{\ld, \th,
\mu}(X_0)<m_1$.

Indeed, suppose not, we have $B_{R_0}(X_0)\cap\Gamma_{1,
\mu}=\varnothing$. Similar arguments as Lemma
\ref{cc3}, one gets $$\sup_{\p B_{R_0}(X_0)}\left(m_1-\psi_{\ld,\th,
\mu}\right)\geq c\sqrt{\Ld+\ld} R(R+R_0),$$ which implies $R_0\leq\f{m_1}{c\sqrt{\Ld+\ld}  R}$.
This is impossible for sufficiently large $R_0$. Hence,
the claim holds.

Without loss of generality, we assume that $\Gamma_{1,
\mu}$ is empty, then we obtain $\psi_{\ld, \th, \mu}<m_1$ in
$\O_{\mu}\cap\{r>R\}$.

In view of the claim, there is a sequence $X_n=(x_n, r_n)\in
\Gamma_{2, \mu}$ such that $R<r_n\leq c$ and
$x_n\rightarrow-\infty$. Hence, there exists a subsequence
$X_{n_k}=(x_{n_k}, r_{n_k})\in \Gamma_{2, \mu}$ and
$r_{n_k}\rightarrow \xi$, $x_{n_k}\rightarrow-\infty$ as
$k\rightarrow+\infty$. Due to Proposition \ref{co2}, we can prove
that $\psi_{\ld, \th, \mu}(X+X_{n_k})\rightarrow\psi_0(X)$ uniformly
in any compact subset of $\{(x, r)| r>R-\xi\}$ as
$k\rightarrow+\infty$, where $\psi_{0}$ is a constant flow with
deflection angle $\th=\pi$. This contradicts with
$\th\in(0,\pi)$. Thus, the free boundaries $\Gamma_{1,\mu}$ and $\Gamma_{2,\mu}$ are non-empty.

{\bf Step 2.} We will verify that $\Gamma_{i,\mu}$ is a connected curve and $x=g_{i,\ld,\th,\mu}(r)$ is a continuous function in $[R,+\infty)$, $i=1,2$.

Without loss of generality, we consider the left free boundary. Let
$(\alpha,\beta)$ be the maximal interval such that
$x=g_{1,\ld,\th,\mu}(r)$ is finite-valued for all $[R,+\infty)$.

Similar arguments as Section 5 in \cite{ACF1}, we obtain
$\alpha=R$, and the limit $\lim_{r\rightarrow
R}g_{1,\ld,\th,\mu}(r)$ exists.

If $\beta<+\infty$, one has
$$x=g_{1,\ld,\th,\mu}(r)\rightarrow+\infty\quad \text{ or}\quad
x=g_{1,\ld,\th,\mu}(r)\rightarrow-\infty\ \ \text{as}\ \
r\rightarrow\beta,$$ which together with Proposition \ref{co2}
implies $\theta=0$ or $\theta=\pi$. This leads a contradiction to
the assumption $\th\in(0,\pi)$.

Therefore, we complete the proof of Lemma \ref{lee}.
\end{proof}

\subsection{Monotonicity with respect to the parameter $\th$}
Next, we will establish a fact that the minimizer $\psi_{\ld, \th,
\mu}$ and free boundary $x=g_{i,\ld,\th,\mu}(r)$ ($i=1,2$) are monotonic with respect to the asymptotic deflection angle $\th$.
\begin{prop}\label{123}Suppose that $\th_1, \th_2\in[0,\pi]$ with $\th_1<\th_2$, $\psi_{\ld,\th_1,\mu}$ and $\psi_{\ld,\th_2,\mu}$ are minimizers to the truncated variational problem ($P_{\ld,\th_1,\mu}$) and ($P_{\ld,\th_2,\mu}$), and $x=g_{i,\ld,\th_1,\mu}(r)$ and $x=g_{i,\ld,\th_2,\mu}(r)$ be the free boundary of $\psi_{\ld,\th_1,\mu}$ and $\psi_{\ld,\th_2,\mu}$, respectively, then \be\label{00}\psi_{\ld, \th_1, \mu}\geq\psi_{\ld, \th_2, \mu}\ \ \text{for $(x,r)\in\O_\mu$},\ee
and\be\label{001}g_{i,\ld,\th_1,\mu}(r)>g_{i,\ld,\th_2,\mu}(r)\ \
\text{for $r\geq R$, $i=1,2$}.\ee
\end{prop}
\begin{proof}
Denote $\psi_1=\psi_{\ld, \th_1, \mu}$ and $\psi_2=\psi_{\ld, \th_2,
\mu}$ for simplicity, and set $v_1=\max\left\{\psi_1,
\psi_2\right\}$ and $v_2=\min\left\{\psi_1, \psi_2\right\}$.

%For  $\th_1=\th_2$, the result $\psi_1=\psi_2$ follows directly from
%the Proposition \ref{bb9}.

For
$\th_1<\th_2$, as is customary Lemma 8.1 in \cite{ACF1}, we obtain $$J_{\lambda, \th_1,
\mu}(\psi_1)=J_{\lambda, \th_1,
\mu}(v_1)\ \ \text{and}\ \ J_{\lambda, \th_2, \mu}(\psi_2)=J_{\lambda, \th_2, \mu}(v_2).$$

Since $\psi_1$ and $\psi_2$ are the minimizers to the
functionals $J_{\ld, \th_1, \mu}$ and $J_{\lambda, \th_2, \mu}$,
respectively, we can now proceed as in Theorem 7.1 in \cite{ACF3} to obtain that
$$\text{either}~~\psi_1\geq\psi_2~~ \text{or}~~\psi_1\leq\psi_2~\text{in}~\O_\mu.
$$

However, noticing that $\psi_1\geq\psi_2$ in $\O_\mu\cap\{r>R_0\}$ for some sufficiently large $R_0>R$, we
conclude that the case $\psi_1\geq\psi_2$ in $\O_\mu$.

Next, without loss of generality, we prove that \eqref{001} holds for $i=1$, namely \be\label{0001}g_{1,\ld,\th_1,\mu}(r)>g_{1,\ld,\th_2,\mu}(r)\ \ \text{for $r\geq R$}.\ee

Indeed, in view of \eqref{00}, one has \be\label{e16} g_{1,
\lambda, \th_1, \mu}(r)\geq g_{1, \lambda, \th_2,
\mu}(r)\ \ \text{for}\ \ r\geq R.\ee

For any $r>R$, suppose not, there exists a point
$X_0=\left(x_0, r_0\right)$ with $r_0>R$ such that
$$x_0=g_{1, \lambda, \th_1, \mu}(r_0)=g_{1, \ld, \th_2,
 \mu}(r_0).$$ Since the free boundary $x=g_{1, \lambda, \th_1, \mu}(r)$ is analytic in $r>R$, and applying Hopf's lemma yields that
$$\f{\p}{\p\nu}\left(\psi_{\lambda, \th_1, \mu}-\psi_{\lambda,
\th_2, \mu}\right)<0\ \ \text{at}~~X_0,$$ where $\nu$ is the
unit outward normal vector of $x=g_{1, \lambda, \th_1, \mu}(r)$
at $X_0$. This contradicts to the free boundary conditions
$\sqrt{\ld+\Ld}=\f{1}{r}\f{\p\psi_{\lambda, \th_1,
\mu}}{\p\nu}<\f{1}{r}\f{\p\psi_{\lambda, \th_2, \mu}}{\p\nu}=\sqrt{\ld+\Ld}$ at
$X_0$.

On another side, for $r=R$, suppose that $g_{1, \lambda,
\th_1, \mu}(R)=g_{1, \ld, \th_2, \mu}(R)$ and
$X_0=(g_{1,\ld,\th_1,\mu}(R), R)$. If $g_{1, \lambda,
\th_1, \mu}(R)\leq-1$, let $G_\delta$ be a domain bounded by $N_1$, $L_1$, $\Gamma_{1, \ld, \th_1,\mu}$
and $\p B_\delta(X_0)$, and $G_\delta\subset\{0<\psi_{\ld,\th_1,\mu}<m_1\}$. If $g_{1, \lambda_\mu,
\th_1, \mu}(R)>-1$, set $G_\delta$ be a domain bounded by $N_1$, $\{(x,R)\mid -1\leq x\leq g_{1, \lambda_\mu,
\th_1, \mu}(R)\}$, $\Gamma_{1, \ld, \th_1,\mu}$
and $\p B_\delta(X_0)$, and $G_\delta\subset\{0<\psi_{\ld,\th_1,\mu}<m_1\}$.

 Set
$$\tilde{\psi}_\e=(1+\e)(m_1-\psi_{\lambda, \th_1, \mu})-(m_1-\psi_{\lambda, \th_2, \mu})\ \ \text{for some $\e>0$}.$$
Recalling the fact $\psi_1\geq \psi_2$ in $\Omega_\mu$, we can choose
$\delta>0$ sufficiently small such that
$$G_\delta\subset\{0<\psi_{\ld, \th_1, \mu}<m_1\}\cap\{0<\psi_{\ld, \th_2, \mu}<m_1\}.$$
It follows from the similar arguments as in Corollary 11.5 \cite{FA1}
that there exists a small $\e>0$ such that
\be\label{e017}\tilde{\psi}_\e<0\  \ \text{in $G_\delta$}.\ee

 Hence, we obtain
$$\f{1}{R}\f{\p\tilde{\psi}_\e\left(g_{1,
\lambda, \th_1, \mu}(R), R\right)}{\p\nu}\geq0,$$ where $\nu$ is
the unit normal vector of the left free boundary $\Gamma_{1,
\ld, \th_1,\mu}$ at $\left(g_{1, \lambda, \th_1, \mu}(R),
R\right)$, then
$$(1+\e)\sqrt{\ld+\Ld}\leq\sqrt{\ld+\Ld}.$$ This leads a contradiction and then the inequality \eqref{0001} holds for $r=R$.

Therefore, we
finish the proof of the Proposition \ref{123}.
\end{proof}

\subsection{Continuous dependence to the parameters $\ld$ and $\th$}

In this subsection, a convergence result to the parameters $\ld$ and
$\th$ will be stated as follows.
\begin{prop}\label{ee1}
For any $\ld>0$ and $\th\in[0,\pi]$, and sequences
$\lambda_n\rightarrow\lambda$, $\th_n\rightarrow
\th$ with $\th_n\in[0,\pi]$, let $\psi_{\ld_n,\th_n,\mu}$ be the minimizer to the variational problem ($P_{\ld_n,\th_n,\mu}$), $x=g_{i,\ld_n,\th_n,\mu}(r)$ and $x=g_{\ld_n,\th_n,\mu}(r)$ be the free boundary of $\psi_{\ld_n,\th_n,\mu}$ and interface, respectively. Then there exist three subsequences still labeled as $\psi_{\lambda_n, \th_n,\mu}$, $g_{i, \lambda_n, \th_n,\mu}(r)$ and $g_{\lambda_n, \th_n,\mu}(r)$ such that \be\label{b51}\psi_{\lambda_n, \th_n, \mu}\rightarrow
\psi_{\lambda, \th,\mu}\ \ \text{weakly
in}~~H_{loc}^1(\O_\mu)~~\text{and pointwise}\ \text{in}\ \O_\mu,\ee
\be\label{b52}g_{i, \lambda_n, \th_n,\mu}(r)\rightarrow g_{i,
\lambda, \th,\mu}(r)\ \ \text{uniformly for $r\geq R$},\ee and \be\label{b052}g_{\lambda_n, \th_n,\mu}(r)\rightarrow g_{
\lambda, \th,\mu}(r)\ \ \text{uniformly for $r\geq 0$}.\ee

Here, $\psi_{\ld,\th,\mu}$ is the minimizer to the variational problem ($P_{\ld,\th,\mu}$) and $x=g_{i,\ld,\th,\mu}(r)$ and $x=g_{\ld,\th,\mu}(r)$ are the free boundary and interface of $\psi_{\ld,\th,\mu}$ for $i=1,2$, respectively.
\end{prop}
\begin{proof}
Firstly, recalling the following facts$$\psi_{\lambda_n,
\th_n,\mu}\in H_{loc}^1(\O_\mu),\ \ \left|\g \psi_{\lambda_n,
\th_n,\mu}\right|\leq C,$$ and using diagonal procedure gives that there
exists a subsequence $\left\{\psi_{\lambda_{n},
\th_n,\mu}\right\}_{n=1}^{\infty}$ and a function $\o\in H_{loc}^1(\O_\mu)$ for
some $0<\alpha<1$ such that
$$\psi_{\lambda_{n}, \th_n,\mu}\rightarrow \o\ \
\text{weakly in}~~H_{loc}^1(\O_\mu),~~C_{loc}^\alpha(\O_\mu)~~\text{
and pointwise in}~~\O_\mu.$$

Along the similar arguments as Lemma 9.2 in \cite{ACF1}, we obtain
that $\o$ is indeed a minimizer to the truncated variational problem
($P_{\ld,\th,\mu}$). Due to the uniqueness of minimizer to the
truncated variational problem ($P_{\ld,\th,\mu}$), we have
$\o=\psi_{\ld,\th,\mu}$. Therefore, we obtain the convergence of
\eqref{b51}.

Secondly, we will show the statement \eqref{b52} for $\Gamma_{1,
\ld_n, \th_n, \mu}$. Indeed, for any $r_n>R$, let
$$X_n=\left(g_{1, \ld_n, \th_n, \mu}(r_n), r_n\right)\in
\Gamma_{1, \ld_n, \th_n, \mu},\ \ \text{and}\ \ X_n\rightarrow
X_0=(x_0,r_0),~~\text{as}~~n\rightarrow+\infty.$$ Then, for any small $r>0$,
the non-degeneracy lemma implies that there exist two positive
constants $C_1$ and $C_2$, such that
$$C_1\ld_nr_n\leq\f{1}{r}\fint_{\p B_r(X_n)}\left(m_1-\psi_{\lambda_n, \th_n,\mu}\right)dS\leq C_2\ld_nr_n.$$
Letting $n\rightarrow+\infty$
gives
$$C_1\ld r_0\leq\f{1}{r}\fint_{\p B_r(X_0)}\left(m_1-\psi_{\lambda, \th,\mu}\right)dS\leq C_2\ld r_0.$$
Moreover, recalling the non-degeneracy Lemma \ref{cc1} and Lemma
\ref{cc2} yields that $X_0\in \Gamma_{1,\mu}$. Hence, we obtain the
assertion \eqref{b52} for $r\in(R, +\infty)$.

Using Lemma 10.4 in \cite{FA1}, we can obtain the result for $r=R$, namely,
$$g_{1,\ld_n,\th_n,\mu}(R)\rightarrow g_{1,\ld,\th,\mu}(R)\ \ \text{as $n\rightarrow+\infty$}.$$

Similarly, \eqref{b52} holds for the right free
boundary $\Gamma_{2,\ld_n,\th_n,\mu}$.

Finally, similar arguments as Theorem 7.1 in \cite{ACF5}, we can obtain \eqref{b052}.
\end{proof}

\subsection{Continuous and smooth fit conditions of the free boundaries}
In this subsection, we will verify that there exist two parameters
$\ld$ and $\th$, such that the free boundaries $\Gamma_{i, \mu}$
connect smoothly at the end points $A_i$ of the nozzles $N_i$
($i=1,2$), respectively. Namely, for any $\mu>0$, there exists a pair of parameters $(\ld_\mu,\th_\mu)$ with $\ld_\mu>0$, $\th_\mu\in(0,\pi)$, such that $$g_{1,\ld_\mu,\th_\mu,\mu}(R)=-1\ \ \text{and}\ \ g_{2,\ld_\mu,\th_\mu,\mu}(R)=1.$$

As already mentioned before, this is the main difference to the
impinging free jet without rigid nozzle walls.

To see this, we first define a set $\Sigma_\mu$
as\be\label{ee5}\Sigma_\mu=\{\ld\mid\ld\geq 0, \text{there exists a
$\th\in(0,\pi)$, such that $g_{1, \lambda, \th, \mu}(R)<-1$ and $
g_{2, \lambda, \th, \mu}(R)>1$}\}.\ee

 The following lemma implies that $\Sigma_\mu$ is non-empty.
\begin{lemma}\label{ee2}There exists $\th_0\in(0,\pi)$ such that
\be\label{000} g_{1, \ld, \th_0, \mu}(R)<-1\ \ \text{and}\ \
g_{2, \ld, \th_0, \mu}(R)>1,\ee for sufficiently small
$\ld>0$.
\end{lemma}
\begin{proof}
For any $\O_0\subset\subset \O_\mu\cap\{r<R\}\cap\{m_2<\psi_{\ld, \th, \mu}<0\}$, firstly, it follows from Lemma 5.2 in \cite{ACF1} that there exists a positive
constant $C$ (depending only on $\O_0$), such that
\be\label{b0001}|\g\psi_{\ld, \th, \mu}|\leq C\ld\ \ \text{in}\ \
\O_0,\ee  provided that $\O_0$ contains a free boundary point.

For $\th\in(0,\pi)$, suppose not, without loss of generality, suppose $g_{2,\ld,\th,\mu}(R)\leq1$.

Indeed, it follows from the monotonicity
of $\psi_{\ld, \th ,\mu}(x,r)$ with respect to $x$, that
there exists a point $X_1\in\O_\mu$, such that
$$\psi_{\ld, \th, \mu}(X_1)=\f{m_2}{2},\ \ \text{with}\ \ X_1=(x_1, R), \ \ \text{and}\ \  g_{\ld,\th, \mu}(R)<x_1<1.$$

Denote $X_2=(x_2, R)$ as the initial point of the right free boundary
$x=g_{2, \ld, \th, \mu}(r)$, due to the monotonicity of $\psi_{\ld, \th,
\mu}$ with respect to $x$, one has $x_2>x_1$. Taking
$X_3=\left(\f{3x_2-x_1}{2}, R\right)$, an arc $\gamma\in\O_{\mu}\cap\{r>R\}$
connecting $X_1$ to $X_3$ and $|\gamma|\leq C|X_2-X_1|=C|x_2-x_1|$,
which intersects $\Gamma_{2, \mu}$ at $X_4$, $\gamma_0$ denotes the
arc part $\gamma$ from $X_1$ to $X_4$.

Let $\O_0$ be bounded by $\gamma_0$, $r=R$ and $\Gamma_{2, \mu}$,
it follows from \eqref{b0001} that \be\label{b0002}|\g\psi_{\ld, \th,
\mu}|\leq C\ld\ \text{in} \ \ \O_0\setminus B_\delta(A_2),\ \ \text{for
sufficiently small $\delta>0$}.\ee

Hence, it follows from \eqref{b0002} that $$-\f{m_2}{2}=\psi_{\ld, \th,
\mu}(X_1)-\psi_{\ld, \th, \mu}(X_4)\leq\int_{\gamma_0}|\g\psi_{\ld, \th,
\mu}|dl\leq C\ld|X_1-X_2|\leq C\ld\left(1-g_{\ld,\th,\mu}(R)\right),$$ which is impossible with sufficiently small $\ld$. Therefore, for some sufficiently small $\ld>0$, one has $g_{2,\ld,\th,\mu}(R)>1$.

Indeed, due to Proposition \ref{ee1} and Remark \ref{e00}, we obtain
that there exists an $\th_0$ ($\pi-\th_0\ll1$) such
that \eqref{000} holds.

Hence, we complete the proof of this lemma.
\end{proof}

Next, the following lemma implies that the set $\Sigma_\mu$ has a uniform positive lower bound.

\begin{lemma}\label{ee3}If $\th\in(0,\pi)$, we have \be\label{b60}\min\left\{-g_{1,
\lambda, \th, \mu}(R), g_{2, \lambda, \th,
\mu}(R)\right\}<1,\ee for sufficiently large $\lambda$.
\end{lemma}
\begin{proof}
Indeed, it suffices to prove that the free boundaries $\Gamma_{1,\mu}: x=g_{1, \ld, \th,\mu}(r)$ and $\Gamma_{2,\mu}: x=g_{2, \ld, \th,\mu}(r)$ with $R\leq
r\leq 2R$ are contained
in a neighborhood of $\Gamma_{\mu}: x=g_{\ld, \th,\mu}(r)$ for sufficiently large $\lambda$.

Firstly, we prove that the free boundary $\Gamma_{1,\mu}$ with $R\leq
r\leq 2R$ is contained
in a neighborhood of $\Gamma_{\mu}$ for sufficiently large $\lambda$.

 Suppose not, then there exists a small and fixed
$r_0>0$ and $\t X=(\t x,\t r)\in\Gamma_{1,\mu}
\cap\left\{R\leq
r\leq 2R\right\}$, and
$B_{r_0}(\t X)\subset \O_\mu\cap\{r>R\}\cap\{0<\psi_{\ld, \th, \mu}<m_1\}$,
such that for any $\lambda>0$,
$$B_{r_0}(\t X)\cap \Gamma_\mu=\varnothing.$$
Thus, the non-degeneracy Lemma \ref{cc2} implies that
$$\sqrt{\ld+\Ld} C\t r\leq\f{1}{r_0}\fint_{\p B_{r_0}(\t X)}\left(m_1-\psi_{\lambda, \th,
\mu}\right )dS\leq\f{m_1}{r_0},$$ which yields
$$\sqrt{\ld+\Ld} \leq\f{m_1}{C r_0R}.$$ This leads to a contradiction for sufficiently large $\ld>0$.

Similarly, we can prove that the free boundaries $\Gamma_{2,\mu}$ with $R\leq
r\leq 2R$ is contained
in a neighborhood of $\Gamma_{\mu}$ for sufficiently large $\lambda$.

Therefore, we finish the proof of Lemma \ref{ee3}.
\end{proof}
Define\be\label{e10}\lambda_\mu=\sup\left\{\ld\mid
\ld\in\Sigma_\mu\right\},\ee Lemma \ref{ee3} implies that there exists a positive constant $C$
independent of $\mu$, such that
$$\lambda_\mu\leq C.$$

Finally, we will check that there exists a $\th_\mu\in(0,\pi)$ such
that\be\label{e11}g_{1, \lambda_\mu,
\th_\mu, \mu}(R)=-1\ \ \text{and}\ \ g_{2, \lambda_\mu, \th_\mu,
\mu}(R)=1.\ee

\begin{prop}\label{ee4}There exists a $\th_\mu\in(0,\pi)$ such that \eqref{e11} holds. Furthermore,
$N_i\cup\Gamma_{i, \ld_\mu, \th_{\mu}, \mu}$ is $C^1$-smooth in a
neighborhood of $A_i$, for $i=1, 2$.\end{prop}
\begin{proof}
Taking a sequence $(\lambda_n,\th_n)$ such that
$$g_{1, \lambda_n, \th_n, \mu}(R)<-1,\ \ g_{2,
\lambda_n, \th_n, \mu}(R)>1,$$ and $$\ld_n\rightarrow\ld_\mu>0,\ \ \th_n\rightarrow\th_\mu\in[0,\pi].$$

Noticing the fact that
$x=g_{i,\ld,\th, \mu}(r)$ $(i=1, 2)$ is continuous with respect to the
parameters $\lambda$ and $\th$, then we have \be\label{e13}g_{1, \lambda_\mu, \th_\mu, \mu}(R)\leq-1\ \
\text{and}\ \ g_{2, \lambda_\mu, \th_\mu, \mu}(R)\geq1.\ee

 Firstly, we claim that \be\label{e14}0<\th_\mu<\pi.\ee
Without loss of generality, we suppose $\th_\mu=0$, then for
any $\t\th>0$, the monotonicity in Proposition \ref{123} gives that
\be\label{e15}\psi_{\lambda_\mu, \t\th,
\mu}\leq\psi_{\lambda_\mu, 0, \mu},\ee and
 \be\label{e17} g_{1,
\lambda_\mu, \t\th, \mu}(r)< g_{1, \lambda_\mu, 0,
\mu}(r)\quad\text{for}\  r\geq R.\ee

It follows from $\th_\mu=0$ that $g_{2,
\ld_\mu, 0, \mu}(R)=+\infty$, choosing $\t\th>0$ be sufficiently small, due to the
convergence of the free boundaries \eqref{b52}, one gets
$$g_{2, \lambda_\mu,\tilde{\th}, \mu}(R)\geq2.$$

Furthermore, \eqref{e17} implies that $$g_{1,
\lambda_\mu, \tilde{\th}, \mu}(R)<-1.$$ Then, we can choose a
$\lambda_0>\lambda_\mu$ and $\lambda_0-\lambda_\mu$ suitably small
such that
$$g_{1, \lambda_0, \tilde{\th}, \mu}(R)<-1\ \ \text{and}\ \ g_{2, \lambda_0,
\tilde{\th}, \mu}(R)>1,$$ which implies $\lambda_0\in\Sigma_{\mu}$.
This contradicts with the definition of $\ld_\mu$.

Consequently, we obtain $\th_\mu>0$. Similarly, we can prove $\th_\mu<\pi$. Hence, the claim \eqref{e14} holds.

Moreover, we will verify the continuous
fit conditions \eqref{e11}. Indeed, suppose not, without loss of generality, we assume that
$$g_{1, \lambda_\mu, \th_\mu, \mu}(R)<-1.$$ Taking $\tilde{\th}\in(0,\th_\mu)$ with $\th_\mu-\t\th$ being suitably small, then the continuity of $g_{1,\ld_\mu, \th_\mu,
\mu}(R)$ with respect to $\th$ gives
$$g_{1, \lambda_\mu, \tilde{\th},
\mu}(R)<-1.$$

Similar to \eqref{e17}, we have $$1\leq g_{2, \lambda_\mu, \th_\mu,
\mu}(R)<g_{2, \lambda_\mu, \tilde{\th}, \mu}(R).$$

Hence, $$g_{1, \lambda_\mu, \tilde{\th}, \mu}(R)<-1\ \ \text{and}\ \ g_{2, \lambda_\mu, \tilde{\th}, \mu}(R)>1.$$

Therefore, similar to the above arguments for $\th=\t\th$, we can choose a $\ld_0>\ld_\mu$ and
$\ld_0-\ld_\mu$ being sufficiently small, and $\ld_0\in\Sigma_\mu$, it
leads a contradiction to the definition of $\ld_\mu$.

Thus, we obtain the continuous fit conditions \eqref{e11}.

Furthermore, the similar proof to the jet flow problem in
\cite{ACF3} implies that the free boundaries are $C^1$-smooth at the
end points of the nozzles $A_i$ $(i=1, 2)$, we omit it here.
\end{proof}
\subsection{Existence of the impinging outgoing jet}

In order to obtain the existence of the impinging outgoing jet, we
take a sequence $\mu=\mu_n\rightarrow+\infty$, and the corresponding
($\lambda_{\mu_n}, \th_{\mu_n}$) with $\ld_{\mu_n}>0$ and
$\th_{\mu_n}\in(0,\pi)$, $$g_{1, \ld_{\mu_n, \th_{\mu_n},
\mu_n}}(R)=-1,\ \ \text{and}\ \ g_{2, \ld_{\mu_n, \th_{\mu_n},
\mu_n}}(R)=1,$$ then there exist a $\ld\geq0$ and $\th\in[0,\pi]$
and a subsequence $\mu_{n}$, such that
$\lambda_{\mu_{n}}\rightarrow\lambda$, $\th_{\mu_{n}}\rightarrow
\th$ and $$\psi_{\lambda_{\mu_{n}}, \th_{\mu_{n}},
\mu_{n}}\rightarrow\psi_{\lambda, \th}\ \ \text{weakly
in}~~H_{loc}^1(\O)~~\text{and a.e in}~~\O.$$

 The similar arguments
as in Proposition \ref{ee1} imply that $\psi_{\lambda, \th}$ is a
local minimizer to the variational problem $J_{\lambda, \th}$, namely,
$$J_{\O_0}(\psi_{\ld, \th})\leq J_{\O_0}(v)\ \ \text{for any $\O_0\subset\subset\O$ and $v-\psi_{\ld, \th}\in H_0^1(\O_0)$},$$
where $J_{\O_0}(v)=\int_{\O_0}r\left|\f{\g\psi}{r}-\left(\sqrt{\Lambda+\ld} \chi_{\{0<\psi<m_1\}}+\sqrt{\ld} \chi_{\{m_2<\psi\leq0\}}\right)e\right|^2dX$.

Furthermore, along the similar arguments in Proposition \ref{bb1}, we can check that $\psi_{\ld,
\th}$ is a weak solution to the boundary value problem \eqref{b3}.

Since
$$m_2\leq \psi_{\ld, \th}\leq m_1\ \ \text{in}~~\O,$$ and
\be\label{e000019}\psi_{\lambda,
\th}(x,r)\geq\psi_{\lambda, \th}(\tilde{x},r)\ \ \ \
\text{for any}~~x<\tilde{x},\ee using the same arguments as before, there exist two $C^1$-smooth functions $x=g_{1, \ld,
\th}(r)$ and $x=g_{2, \ld, \th}(r)$ such that
\be\label{e019}g_{1,\ld_{\mu_{n}}, \th_{\mu_{n}},
{\mu_{n}}}(r)\rightarrow g_{1,\ld,\th}(r)\ \ \text{for any $r\in[R, +\infty)$},\ee and \be\label{e0019}
g_{2,\ld_{\mu_{n}}, \th_{\mu_{n}}, {\mu_{n}}}(r)\rightarrow
g_{2,\ld,\th}(r)\ \ \text{for any $r\in[R, +\infty)$}.\ee
  \be\label{e19} g_{1,\ld,\th}(R)=-1,\quad\text{and}\quad g_{2,\ld,\th}(R)=1.\ee

Furthermore, along the similar arguments as Lemma \ref{ee2} and
\ref{ee3}, we assert that \be\label{e021}\ld>0\ \ \text{and}\ \
0<\th<\pi,\ee and the smooth fit condition of $x=g_{i,\ld,\th}(r)$
at $A_i$ follows immediately from the arguments in Proposition
\ref{ee4} for $i=1,2$.

Using the standard elliptic estimates yields that $\psi_{\ld,
\th}\in C^{2, \sigma}(\O_1\cup\O_2)\cap C^{0}(\overline{\O_1\cup\O_2})$
for some $\sigma\in(0,1)$ and it solves the boundary value problem \eqref{b3}.

 Hence, the existence of the impinging outgoing jets in Theorem \ref{thm1} has been established.

Next, we will show the positivity of radial velocity to the
axially symmetric impinging outgoing jets.
\begin{prop}\label{ff1} Let $\psi_{\ld,\th}$ be the solution to the boundary value problem \eqref{b3}, then
\be\label{f2}m_2<\psi_{\lambda, \th}<m_1\ \ \text{in} ~~G,\ee and
\be\label{f3}V= -\f{1}{r}\f{\p\psi_{\ld,\th}}{\p x}>0\ \ \
\text{in}~~\overline{G}\setminus \left(N_0\cup\Gamma\right),\ee where $G$ is bounded by
$N_i$, $\Gamma_i$ and $N_0$ for $i=1,2$.
\end{prop}
\begin{proof}
Noting that $$\Delta\psi_{\ld, \th}-\f{1}{r}\f{\p\psi_{\ld,\th}}{\p r}=0\ \text{in any bounded connected
smooth open subdomain $G_0\subset G\setminus\Gamma$},$$ and $m_2\leq\psi_{\ld,\th}\leq m_1$ on $\p G_0$, then, the strong maximum principle
implies that
$$m_2<\psi_{\ld, \th}<m_1\ \ \text{in}~~G_0.$$ The arbitrariness of domain
$G_0\subset G$ yields to \eqref{f2}.

Next, since $N_1\cup\Gamma_1\in C^1$, there exists
a bounded smooth subdomain $G_0\subset \{0<\psi_{\ld,\th}<m_1\}$ with
$\overline{G_0}\cap N_1=X_0$ (or $G_0\subset \{m_2<\psi_{\ld,\th}<0\}$ with
$\overline{G_0}\cap N_2=X_0$). Then, $w=-\f{\p\psi_{\ld, \th}}{\p x}$ satisfies
$$\Delta w-\f{1}{r}\f{\p w}{\p r}=0\ \ \ \text{in}~~G_0.$$

Since $\psi_{\ld, \th}=m_1$ on $N_1$, the slip boundary
condition \eqref{a5} implies that
$$\p_{x}\psi_{\ld, \th}(f_1(r), r)f_1'(r)+\p_{r}\psi_{\ld, \th}(f_1(r), r)=0.$$
This implies that the unit outward normal
derivative satisfies
$$\f{\p \psi_{\ld, \th}}{\p \nu}\left(f_1(r), r)\right)=-\p_{
x}\psi_{\ld, \th}\left(f_1(r), r)\right)\sqrt{1+f_1'(r)^{2}}.$$ On another
hand, $\psi_{\ld, \th}$ attains its maximum on $N_1$, and
Hopf's lemma gives that
$$w=-\p_{x}\psi_{\ld, \th}>0,\ \ \ \text{on}~~N_1.$$

Similarly, one gets $$w=-\p_{x}\psi_{\ld, \th}>0,\ \ \ \text{on}~~N_2.$$

Then, in view of \eqref{e000019}, one has
$$w\geq0\ \ \ \text{on}~~\p G_0,\ \ \text{and}\ \ w(X_0)>0,$$
and applying the maximum principle to $w=-\f{\p\psi_{\ld,\th}}{\p x}$
yields that
$$ w>0\ \ \
\text{in any subdomain}~~G_0\subset G.$$

Finally, we claim that \eqref{f3} holds on $\Gamma_1\cup\Gamma_2$.

Recalling the fact $$w=\f{\sqrt{\ld+\Ld} r}{\sqrt{1+(g_{1,\ld,\th}'(r))^2}}\geq0\ \ \text{on}~~\Gamma_1,\ \ \text{and}\ \ w
=\f{\sqrt{\ld} r}{\sqrt{1+(g_{2,\ld,\th}'(r))^2}}\geq0\ \ \text{on}~~\Gamma_2,$$

Suppose that the claim can not hold, then, without loss of
generality, there exists a $r_0\geq R$ such that
$g_{1,\ld,\th}'(r_0)=+\infty$ or $-\infty$,
$\p_{x}\psi_{\ld, \th}(g_{1,\ld,\th}(r_0),r_0)=0$ and $w(g_{1,\ld,\th}(r_0),r_0)=0$.

Thanks to the fact that $\f{|\g\psi_{\ld,\th}|}{r}=\sqrt{\ld+\Ld}$
on $\Gamma_1$ (or $\f{|\g\psi_{\ld,\th}|}{r}=\sqrt{\ld}$ on
$\Gamma_2$), then
$$\f{\p}{\p s}\left(\f{|\g\psi_{\ld,\th}|^2}{r^2}\right)=0\ \ \text{on $\Gamma_1\cup\Gamma_2$},$$ where $s=(1,0)$ is the tangential vector of $\Gamma_1\cup\Gamma_2$. This implies that $$\left(\f{\p}{\p x}\left(\f{|\g\psi_{\ld,\th}|^2}{r^2}\right),\f{\p}{\p r}\left(\f{|\g\psi_{\ld,\th}|^2}{r^2}\right)\right)\cdot\left(1,0\right)=0\ \ \text{on $\Gamma_1\cup\Gamma_2$}.$$

Then, one has \be\label{f03}\p_{xr}\psi_{\ld,\th}(g_{1,\ld,\th}(r_0),r_0)=\p_rw(g_{1,\ld,\th}(r_0),r_0)=0.\ee

However, the Hopf's Lemma gives that $$\left|\f{\p w}{\p \nu}\right|=\left|\f{\p w}{\p
r}\right|>0\ \ \text{at $(g_{1,\ld,\th}(r_0),r_0)$},$$ which contradicts
with \eqref{f03}. Then, we prove that \eqref{f3} holds on
$\Gamma_1\cup\Gamma_2$.

Therefore, we obtain the positivity of the vertical velocity and complete
the proof of Proposition \ref{ff1}.
\end{proof}

\subsection{The properties of the interface}In this subsection, we will show that there exists a $C^1$-smooth curve $\Gamma: \{\psi_{\lambda,\th}=0\}\cap\{r>0\}$
separating the two fluids, and the axially symmetric impinging outgoing jet
established here possesses a unique branching point on the symmetric
axis $N_0$. For $\Ld=0$, the proof is similar to Section 4.10 in \cite{DW}, we omit it here.

Next, it suffices to prove that the results hold for $\Ld>0$.

Indeed, taking subsequence $\mu_n\rightarrow+\infty$, one has
$$g_{\ld_{\mu_{n}}, \th_{\mu_{n}}, {\mu_{n}}}(r)\rightarrow
g_{\ld,\th}(r)\ \ \text{for any
$r>0$}.$$
Then, the similar arguments as Lemma \ref{dd02} implies that
$x=g_{\ld,\th}(r)\in[-\infty, +\infty]$ is generalized continuous function in
$[0,+\infty)$, we need to prove that $x=g_{\ld,\th}(r)$ is finite valued for
any $r\in[0,+\infty)$.\be\label{e00019}
g_{\ld_{\mu_{n}}, \th_{\mu_{n}}, {\mu_{n}}}(r)\rightarrow
g_{\ld,\th}(r)\ \ \text{for any $r\in(0, +\infty)$}.\ee

Since $g_{\ld,\th}(r)<g_{2,\ld,\th}(r)$ for any $r>R$ and $g_{\ld,\th}(r)>g_{1,\ld,\th}(r)$ for any
$r>R$, then it suffices to prove that $g_{\ld,\th}(r)$ is finite valued for
$0\leq r \leq M_0$, where $\max\{r_1,r_2\}\leq M_0\leq R$. Denote by
$(\alpha_i, \beta_i)\subset[0,M_0]$ ($i=1,2,...$,
$\alpha_i<\beta_{i}$ and $\beta_i\leq\alpha_{i+1}$) the maximum
intervals where $x=g_{\ld,\th}(r)$ is finite valued.

Similar arguments as Proposition 5.1 in \cite{WX}, we can prove that the number of intervals
$(\alpha_i,\beta_i)$ is one, denote $(\alpha, +\infty)$ for simplicity.

Next, we will prove that $\alpha=0$.

Suppose $\alpha>0$ and $\lim_{r\rightarrow \alpha}g_{\ld,\th}(r)=-\infty$. For
some sufficiently large $M>0$, Set
$$\begin{aligned}&R_0=\max\{x\mid x=g_{\ld,\th}(r), \alpha<r<r_1\}, r_3=\min\{r\mid
g_{\ld,\th}(r)=R_0-M\}, \\& r_4=\max\{r\mid g_{\ld,\th}(r)=R_0, \alpha<r<r_1\}, \text{such
that $R_0-M\leq g_{\ld,\th}(r)\leq R_0$\ \ with $r_3\leq r\leq
r_4$}.\end{aligned}$$It follows from Lemma 6.1 in \cite{ACF5} that
$$M\leq C(r_4-r_3)\leq Cr_1,$$ and this is impossible when $M$
sufficiently large.

Furthermore, the case $\alpha>0$ and
$\lim_{r\rightarrow\alpha}g_{\ld,\th}(r)=+\infty$, there exists a sufficiently
large $R_0>0$ and for any $M>0$, define
$$\begin{aligned}&H_{R_0}=\min\{r\mid g_{\ld,\th}(r)=R_0\},\ \
l_{R_0}=\left\{(R_0, r)\mid 0\leq r \leq H_{R_0}\right\},\\
&H_{R_0+M}=\min\{r\mid g_{\ld,\th}(r)=R_0+M\}, \ \ \text{and}\ \
l_{R_0+M}=\left\{(R_0, r)\mid 0\leq r \leq
H_{R_0+M}\right\}.\end{aligned}$$ We define a domain $G_{R_0,M}$,
which is bounded by $l_{R_0}$, $l_{R_0+M}$, $\Gamma$ and $x$-axis.

Applying Green's formula in $G_{R_0,M}$ and $|\psi_{\ld,\th}|\leq Cr^2$ in $\O\cap\{r<\min\{r_1,r_2\}\}$ (using the fact \eqref{b18}), we
find that
\be\label{dd1}\begin{aligned}-\int_{\p G_{R_0,M}}\f{x-R_0}{r}\f{\p\psi_{\ld,\th}}{\p \nu}dS&=-\int_{\p G_{R_0,M}}\f{\p(x-R_0)}{\p\nu}\f{\psi_{\ld,\th}}{r}dS\\&=\int_{\p G_{R_0,M}\cap\{x=R_0\}}\f{\psi_{\ld,\th}}{r}dS-\int_{\p G_{R_0,M}\cap\{x=R_0+M\}}\f{\psi_{\ld,\th}}{r}dS\\ &\leq CH_{R_0}^2,\end{aligned}\ee where $\nu$ is the unit normal
vector.

Indeed, in view of $\f{|\g\psi_{\ld,\th}^+|}{r}\geq \sqrt{\Ld}$ on
$\Gamma\cap\{R_0\leq x\leq R_0+M\}$, and $\f{\p\psi_{\ld,\th}}{\p
r}\geq0$ on $x$-axis with $x\geq R_0$, the left hand side of \eqref{dd1} is estimated as $$\begin{aligned}-\int_{\p G_{R_0,
M}}\f{x-R_0}{r}\f{\p\psi_{\ld,\th}}{\p \nu}dS&=-\int_{\p
G_{R_0,
M}\cap(\Gamma\cup\{r=0\}\cup\{x=R_0+M\})}\f{x-R_0}{r}\f{\p\psi_{\ld,\th}}{\p\nu}dS\\&\geq\sqrt{\Ld}\int_{\p G_{R_0,
M}\cap\Gamma}(x-R_0)dS-M\int_{\p
G_{R_0,
M}\cap\{x=R_0+M\}}\f{1}{r}\f{\p\psi_{\ld,\th}}{\p\nu}dS\\&\geq
c\sqrt{\Ld}M^2-CMH_{R_0},\end{aligned}$$ due to \eqref{b17}. Here $\nu$ is parallel to $\g\psi_{\ld,\th}$ on
free streamlines, the positive constants $c$, $C$ are independent of $M$ and $H_{R_0}$.

This together with \eqref{dd1} gives that
$$M\leq CH_{R_0},$$ then we derives a contradiction for sufficiently
large $M$.

Hence, we obtain $\alpha=0$.

Finally, it suffices to prove that $|g_{\ld,\th}(0+0)|<+\infty$.

Suppose not, without loss of generality, we assume
$g_{\ld,\th}(0+0)=+\infty$. Similar to the above arguments, we only
need to construct domain $G$ bounded by $x=R_0$ with some $R_0>0$
sufficiently large, $x=R_0$, $x=R_0+M$, $r=0$ and interface $\Gamma$,
then $$M\leq C,$$ which derives a contradiction for sufficiently
large $M$.

Similarly, we can exclude that $g_{\ld,\th}(0)=-\infty$. Therefore,
we conclude that $x=g_{\ld,\th}(r)$ is finite for any
$r\in[0,+\infty)$.

Now, collecting all results obtained above, we complete the proof
of Theorem \ref{thm1}.

\section{Uniqueness of the impinging outgoing jet}

In this section, we will investigate the uniqueness of the impinging
outgoing jet and the parameters when $\Lambda=0$.

{\bf Proof of Theorem \ref{thm2}.} Let $\psi_{\lambda,\th}$ and $\t\psi_{\lambda,\th}$ be two solutions to the boundary value problem \eqref{b3}, and
$\Gamma_{i,\lambda,\th}:x=g_{i,\ld,\th}(r)$ and
$\t\Gamma_{i,\lambda,\th}:x=\t g_{i,\ld,\th}(r)$ be the corresponding free
boundaries for $i=1,2$. Due to the continuous fit conditions, one has $$f_1(R)=g_{1,\ld,\th}(R)=\t g_{1,\ld,\th}(R)\ \ \text{and}\ \ f_2(R)=g_{2,\ld,\th}(R)=\t g_{2,\ld,\th}(R).$$

Without loss of generality, we assume $$\lim_{r\rightarrow+\infty}(g_{1,\ld,\th}(r)-\t g_{1,\ld,\th}(r))\geq0.$$ Set $\psi^\e=\psi_{\ld,
\th}(x+\e,r)$ for some $\e\geq0$ and choose a smallest $\e_0\geq0$ such
that
$$\psi^{\e_0}\leq\tilde{\psi}_{\ld, \th}\ \ \text{in $\O$ and $\psi^{\e_0}(X_0)=\tilde{\psi}_{\ld,
\th}(X_0)$},$$ for some $X_0\in\overline{\{m_2<\t\psi_{\ld,\th}<m_1\}}$.

We claim that
$$X_0\notin\{m_2<\psi^{\e_0}<m_1\}\cap\{m_2<\tilde{\psi}_{\ld,\th}<m_1\}.$$
Suppose not and there exists a point
$X_0\in\{m_2<\psi^{\e_0}<m_1\}\cap\{m_2<\tilde{\psi}_{\ld,\th}<m_1\}$, such that
$$m_2<\psi^{\e_0}(X_0)=\t\psi_{\ld, \th}(X_0)<m_1.$$

The continuity of $\psi_{\ld, \th}$ in
$\O$ implies that there exists a ball $B_r(X_0)\subset\{m_2<\psi^{\e_0}<m_1\}\cap\{m_2<\t\psi_{\ld, \th}<m_1\}$ such that
 \be\label{b48}
\left\{\ba{lll} \Delta \t\psi_{\ld, \th}-\f{1}{r}\f{\p \t\psi_{\ld,\th}}{\p r}=0,\ \ \Delta \psi^{\e_0}-\f{1}{r}\f{\p \psi^{\e_0}}{\p r}=0\ \ \ \ &\text{in}\ \  B_r(X_0), \\
\psi^{\e_0}(X)\leq\t\psi_{\ld, \th}(X)\ \ \ \ &\text{on}\  \ \p
B_r(X_0).
 \ea \right.\ee
Therefore, it follows from the strong maximum principle that
$$\psi^{\e_0}(X)=\t\psi_{\ld, \th}(X)\ \ \text{in}\ \
B_r(X_0).$$ Applying the strong maximum principle in $\O$ again, we obtain a contradiction to the boundary condition of $\t\psi_{\ld,\th}$.

Then, one has $$\psi^{\e_0}(X_0)=\tilde{\psi}_{\ld, \th}(X_0)=m_1,\
\ \text{or}\ \  \psi^{\e_0}(X_0)=\tilde{\psi}_{\ld, \th}(X_0)=m_2.$$

 Hence, the following two cases may occur.

{\bf Case 1.}  $\e_0>0$, then $X_0\in \Gamma_{1, \ld,
\th}\cap\tilde{\Gamma} _{1, \ld, \th}$ or $X_0\in \Gamma_{2, \ld,
\th}\cap\tilde{\Gamma}_{2, \ld, \th}$ and $X_0\neq A_1$, $A_2$.
Then,
$$\Delta\psi^{\e_0}-\f{1}{r}\f{\p\psi^{\e_0}}{\p r}=\Delta\tilde{\psi}_{\ld,\th}-\f{1}{r}\f{\p\t\psi_{\ld,\th}}{\p r}=0\ \
\text{in}\ \O\cap\{m_2<\psi^{\e_0}<m_1\}\cap\{m_2<\tilde{\psi}_{\ld,\th}<m_1\}.$$
The $C^1$-smoothness of the free boundaries implies that
$\Gamma_{1,\ld,\th}$ (or $\Gamma_{2,\ld,\th}$) is tangent to
$\tilde{\Gamma}_{1,\ld,\th}$ (or $\tilde{\Gamma}_{2,\ld,\th}$) at the point
$X_0$. Then, it follows from the maximum principle that
$$\sqrt{\ld}=\f1r\f{\p \psi^{\e_0}}{\p \nu}>\f1r\f{\p\tilde{\psi}_{\ld, \th}}{\p
\nu}=\sqrt{\ld} \ \ \text{or}\ \ \sqrt{\ld}=\f1r\f{\p \psi^{\e_0}}{\p
\nu}<\f1r\f{\p\tilde{\psi}_{\ld, \th}}{\p \nu}=\sqrt{\ld} \ \ \text{at}\ \ X_0,$$
where $\nu$ is outer normal vector, we derive a contradiction.

{\bf Case 2.} $\e_0=0$, then $X_0=A_1$ or $A_2$. Without loss of
generality, suppose $X_0=A_1$, similar to the proof of Proposition
\ref{ee4}, construct a domain $G_\delta$ with $\delta>0$ and
$\bar{\psi}=(1+\zeta)(m_1-\tilde{\psi}_{\ld,\th})-(m_1-\psi^{\e_0})$
as \eqref{e017}, then for some sufficiently small $\zeta>0$, which gives $$\bar\psi<0\ \ \text{in $G_\delta$},$$ we have
$$(1+\zeta)\sqrt{\ld}=\f{1+\zeta}{r}\left|\f{\p\tilde{\psi}_{\ld, \th}}
{\p \nu}\right|\leq\f1r\left|\f{\p \psi^{\e_0}}{\p
\nu}\right|=\sqrt{\ld}\ \ \text{at}\ \ A_1,\ \ \text{for small
$\zeta>0$},$$ a contradiction. Similarly, we obtain $X_0\neq A_2$.

Hence,  we obtain the uniqueness of the minimizer $\psi_{\ld,\th}$ for given $\ld$ and $\theta$.

Next, we will prove $\th=\t\th$ for given $\ld$.

Suppose not, without loss of generality, we assume $\th<\t \th$.

Let $\psi_{\ld,\th}$ and $\psi_{\ld,\tilde{\th}}$ be the two solutions
to the boundary value problem \eqref{b3} corresponding to the
pairs of the parameters $(\ld, \th)$ and $(\ld, \tilde{\th})$, respectively. Due to Proposition \ref{123}, one has $$\psi_{\ld,\th}(X)\geq\psi_{\ld, \tilde{\th}}(X)\ \ \text{in $\O$}.$$

Similar arguments as above, we take $X_0=A_1$ and derive a contradiction.

Hence, we obtain $\th=\tilde{\th}$ as desired.

\section{Asymptotic behavior of impinging outgoing jet}

In this section, we will establish the asymptotic behaviors of axially symmetric impinging outgoing jets in far fields which are
stated in Theorem \ref{thm4}.

{\bf Proof of Theorem \ref{thm4}.} Due to the standard elliptic estimates, there exists a
constant $C$ depending only on $m_1$, $m_2$ and $\ld$ such that
\be\label{f6}\|\g \psi_{\ld, \th}\|_{C^{1, \sigma}(G)}\leq C,\ \
\text{for some}\ \ 0<\sigma<1,\ee where $G\subset\subset\{0<\psi_{\ld, \th}<m_1\}\cup\{m_2<\psi_{\ld, \th}<0\}$.

Set $\psi_{n}(x, r)=\psi_{\ld, \th}(x-n, r)$ and a strip
$E=\left\{-\infty<x<+\infty\right\}\times\left\{0<r<r_1\right\}$,
there exists a subsequence still labeled as $\psi_n(x,r)$ such that
$$\psi_n(x, r)\rightarrow \psi_0(x, r),\ \
\text{uniformly in}~~C^{2, \sigma_0}(S),~~0<\sigma_0<\sigma,$$ for any
compact set $S\subset\subset E$
 and $\psi_0(x, r)$ solves the following boundary value problem in the strip $E$,
 \be\label{f7}
\left\{\ba{lll}
\Delta \psi_0(x, r)-\f{1}{r}\f{\p\psi_0}{\p r}=0,\ \ \ \ \ \ &\text{in}\ \  E, \\
\psi_0\left(x, 0\right)=0,\ \ \psi_0\left(x, r_1\right)=m_1,\ \ \ \
\ \ \ \ \
&\text{for}\ \ -\infty<x<+\infty, \\
\max\left\{\f{m_2}{r_2^2}r^2, m_2\right\}\leq\psi_0(x, r)\leq
\f{m_1}{r_1^2}r^2, &\text{in}\ \  E,
 \ea \right.\ee where we have used the Lemma \ref{cc9}. Obviously, the problem \eqref{f7} has a unique solution
 as
\be\label{f8}\psi_0(x,r)=\f{m_1}{r_1^2}r^2, \ \ r\in[0, r_1].\ee

Hence, we obtain $$\g\psi_{\ld, \th}(x, r)\rightarrow\left(0,
\f{2m_1r}{r_1^2}\right)\ \ \text{in $C^{1,\sigma_0}(S)$ as}~~x\rightarrow-\infty.$$

Using the Bernoulli's law yields the asymptotic behavior
\eqref{a14} and \eqref{a15} of flow field in the upstream.

Along the similar arguments as before, we obtain the asymptotic behavior in the upstream

 $$\g\psi_{\ld, \th}(x, r)\rightarrow\left(0,
\f{2m_2r}{r_2^2}\right)\ \ \text{in $C^{1,\sigma_0}(S')$ as}~~x\rightarrow+\infty,$$ where $S'\subset\subset E'=\left\{-\infty<x<+\infty\right\}\times\left\{0<r<r_2\right\}$.

Finally, it follows from Lemma \ref{lem3}, we obtain \eqref{a01}, \eqref{a09} and \eqref{a009}.
Hence, the proof of Theorem \ref{thm4} is done.

\section{Appendix}
The minimizer $\psi_{\ld,\th,\mu}$ satisfies the following elliptic equation in a weak sense. The similar proofs of these results can be found in Theorem 2.2-2.3 in \cite{ACF4}, we omit here.
\begin{prop}\label{bb1}
Let $\psi_{\ld, \th,\mu}$ be a minimizer to the truncated variational
problem ($P_{\ld,\th,\mu}$), and $\mathfrak{L}^2\left(\{\psi_{\ld, \th,\mu}=0\}\right)=0$ ($\mathfrak{L}^2$ is the two dimensional Lebesgue measure), then \be\label{b7}\Delta\psi_{\ld, \th,\mu}-\f{1}{r}\f{\p\psi_{\ld,\th,\mu}}{\p r}=0,\ \
\text{in}\ \ \O_\mu\cap\left\{m_2<\psi_{\ld, \th,\mu}<m_1\right\}\cap\{\psi_{\ld, \th,\mu}\neq0\},
\ee and \be\label{b07}\Delta\psi_{\ld, \th,\mu}-\f{1}{r}\f{\p\psi_{\ld,\th,\mu}}{\p r}\geq0,\ \
\text{in}\ \ D_{\mu}=\O_\mu\cap\{r<R\},
\ee in a weak sense.
\end{prop}

\bigskip

\
{\bf Acknowledgments.} The authors would like to thank the referees
for their helpful suggestions and careful reading which has improved the presentation of this paper.
\

{\bf Conflict of interest. } The authors declare that they have no
conflict of interest.

\

\bibliographystyle{plain}

\end{document}